\documentclass[11pt]{article}
\usepackage{amsmath}
\usepackage{amssymb,amsbsy,amsthm}
\usepackage[dvips]{graphicx}
\usepackage{caption}
\usepackage{subcaption}
\usepackage{pstricks}
\usepackage{pst-eps}
\usepackage{pst-node}
\usepackage{enumerate}
\usepackage{dsfont}
\usepackage[margin=1in]{geometry}
\usepackage{pdfsync}
\usepackage{hyperref}
\usepackage{verbatim}
\usepackage{comment}

\allowdisplaybreaks[1]

\DeclareMathOperator{\G}{\mathbb{G}} %geodesic graph
\newcommand{\Zz}{\mathbb{Z}^d} %lattice in the plane
\newtheorem{theorem}{Theorem}
\newtheorem{lemma}[theorem]{Lemma}
\newtheorem{proposition}[theorem]{Proposition}
\newtheorem{claim}[theorem]{Claim}

\newtheorem{corollary}[theorem]{Corollary}

\newtheorem{definition}{Definition}

\newtheorem{remark}{Remark}

\numberwithin{equation}{section}
\numberwithin{theorem}{section}
\title{Absence of backward infinite paths for first-passage percolation in arbitrary dimension}

\author{Gerandy Brito\footnote{College of Computing, Georgia Institute of Technology, 801 Atlantic Dr., Atlanta, GA 30332.}, Michael Damron\footnote{School of Mathematics, Georgia Institute of Technology, 686 Cherry St., Atlanta, GA 30332.}, and Jack Hanson\footnote{Department of Mathematics, City College of New York, 160 Convent Ave., New York, NY 10031 and the Graduate Center, CUNY, 365 Fifth Avenue New York, NY 10016.}}

\date{}

\begin{document}

\maketitle

\abstract{In first-passage percolation (FPP), one places nonnegative random variables (weights) $(t_e)$ on the edges of a graph and studies the induced weighted graph metric. We consider FPP on $\mathbb{Z}^d$ for $d \geq 2$ and analyze the geometric properties of geodesics, which are optimizing paths for the metric. Specifically, we address the question of existence of bigeodesics, which are doubly-infinite paths whose subpaths are geodesics. It is a famous conjecture originating from a question of Furstenberg and most strongly supported for $d=2$ that for continuously distributed i.i.d.~weights, there a.s.~are no bigeodesics. We provide the first progress on this question in general dimensions under no unproven assumptions. Our main result is that geodesic graphs, introduced in a previous paper of two of the authors, constructed in any deterministic direction a.s.~do not contain doubly-infinite paths. As a consequence, one can construct random graphs of subsequential limits of point-to-hyperplane geodesics which contain no bigeodesics. This gives evidence that bigeodesics, if they exist, cannot be constructed in a translation-invariant manner as limits of point-to-hyperplane geodesics.}

\section{Introduction}

%{\color{green} Stuff to do:
%\begin{enumerate}
%\item make sure we refer to $H_\rho(\alpha)$ in the appendix (and elsewhere) instead of $H_{\hat{\varrho}}(\alpha)$.
%%%\item are the notations $C(0)$ and $C_0$ conflicting (or does $C_0$ only appear as $C_0^b$)?
%%%\item unbounded case sketch (couple of paragraphs)
%%%\item go through appendix. mention it in the organization of paper, or more heavily in introduction?
%%%\item outline of proof (sketch)
%\end{enumerate}
%}

First-passage percolation was introduced by Hammersley-Welsh \cite{HW} in 1965 as a model for fluid flow in a porous medium. Since then, it has been studied in other ways:~as a random growth model, a particle system, or a random metric space (see \cite{ADH17} for a recent survey). It is this last interpretation we take here, and our main focus will be on geometric properties of optimizing paths in this metric. To define the model, let $(t_e)_{e \in \mathcal{E}^d}$ be a collection of nonnegative random variables (weights) indexed by the edge set $\mathcal{E}^d$ of the $d$-dimensional nearest-neighbor lattice $\mathbb{Z}^d$, with $d \geq 2$. The collective distribution of these weights is usually taken to satisfy certain technical conditions, like translation invariance or an i.i.d.~assumption; these will be described in conditions {\bf A} and {\bf B} in Section~\ref{sec: main_results}. A finite (lattice) path from a vertex $x$ to a vertex $y$ is a sequence of alternating vertices and edges $x=v_0, e_0, v_1, \dots, v_{n-1}, e_{n-1}, v_n=y$ such that $e_i$ has endpoints $v_i$ and $v_{i+1}$ for all $i$. Given such a path $\Gamma$, we define its passage time by $T(\Gamma) = \sum_{i=0}^{n-1} t_{e_i}$, and the (first-)passage time from a vertex $x$ to a vertex $y$ is defined as
\[
T(x,y) = \inf_{\Gamma: x \to y} T(\Gamma),
\]
where the infimum is over all lattice paths from $x$ to $y$. 
%(This definition is extended in the obvious way to subsets $A,B$ of $\mathbb{R}^d$ by taking all paths starting in $A$ and ending in $B$; if there is no such path, the passage time is $\infty$.) 

The function $T:\mathbb{Z}^d \times \mathbb{Z}^d \to [0,\infty)$ is readily seen to be symmetric and to satisfy the triangle inequality, and therefore $(\mathbb{Z}^d, T)$ is a pseudometric space. If there are no zero weights in a particular weight configuration, then $T$ is furthermore a metric.

\subsection{Infinite geodesics in FPP}

Our focus is on geodesics, or optimizing paths in the definition of $T$. For general edge-weights (even ones that are i.i.d.) it is still an interesting open problem to determine exactly for which distributions geodesics exist between any pair of vertices $x,y$, but our conditions will guarantee that $\lim_{n \to \infty} T(0,\mathbb{Z}^d \setminus [-n,n]^d) = \infty$ a.s., and this implies a.s.~existence of geodesics from \cite[Prop.~4.4]{ADH17}. Of particular interest are infinite geodesics:
\begin{definition}
An infinite geodesic is an infinite path whose finite subpaths are geodesics. They come in two varieties:
\begin{enumerate}
\item A unigeodesic (geodesic ray) is indexed by $\mathbb{N}$, and
\item a bigeodesic (geodesic line) is indexed by $\mathbb{Z}$.
\end{enumerate}
\end{definition}

How many infinite geodesics are there? If we define $N$ to be the maximal number of vertex-disjoint unigeodesics, then $N$ is seen to be translation invariant, so under an ergodicity assumption on $(t_e)$, it is a.s.~constant. It is not difficult to show that $N \geq 1$ a.s.: let $\Gamma_n$ be any geodesic from $0$ to $ne_1$. There must be an edge incident to $0$ that infinitely many of the $\Gamma_n$'s take first, and then an edge incident to that one that infinitely many of these $\Gamma_n$'s take second, and so on. In this way, we find a subsequential limit $\Gamma$ which is an infinite geodesic. It is considerably more difficult to prove that $N \geq 2$, and this was done in increasing generality by H\"aggstr\"om-Pemantle \cite{HP}, Garet-Marchand \cite{GM}, and Hoffman \cite{H05}. Hoffman's work introduced Busemann functions to FPP, which he used later in \cite{H08} to improve this bound to $N \geq 2d$. This is the best result currently known, and is far from the prediction of $N = \infty$ in the i.i.d.~case that comes from Newman's methods \cite{N95}, which use unproven assumptions about the limiting shape. (See \cite{BH} for an example of a non-i.i.d.~distribution for which $N < \infty$.)

As for bigeodesics, the following was communicated by Kesten in \cite[p.~258]{aspects}

\medskip
\noindent
{\bf Question} (H. Furstenberg)
%\begin{question}[H. Furstenberg]
Do there exist doubly infinite geodesics through 0?
%\end{question}

\medskip
It is widely believed that for $d=2$, a.s.~there are are no bigeodesics when weights $(t_e)$ are i.i.d. and have a common continuous distribution. There are now a few pieces of evidence:
\begin{enumerate}
\item Existence of bigeodesics is equivalent to existence of nonconstant ground states for a disordered ferromagnetic Ising model with couplings distributed as $t_e$. See \cite[Sec.~4.5.2]{ADH17} for discussion and \cite{FLN} for physics arguments against the existence of nonconstant ground states.
\item There are heuristic arguments of C. Newman (see \cite[Sec.~4.5.1]{ADH17}) indicating that in any dimensions in which the ``wandering exponent'' $\xi$ satisfies $\xi>1/2$, a.s.~there are no bigeodesics. This $\xi$ is defined roughly as the number such that the maximal Euclidean distance $D_n$ between any geodesic from $0$ to $ne_1$ and the $e_1$-axis satisfies $\mathbb{E}D_n \sim n^\xi$. One expects $\xi>1/2$ at least in low enough dimensions, and for $d=2$, one expects $\xi=2/3$. (Also see the recent preprint \cite{Alexander} for work on nonexistence of bigeodesics under the unproved condition $\chi >0$, where $\chi$ is the ``fluctuation exponent.'')
\item Two preprints (by Basu-Hoffman-Sly \cite{BHS} and Bal\'azs-Busani-Sepp\"al\"ainen \cite{BBS}) have recently been posted which state results of nonexistence of bigeodesics in a related model, exactly solvable last-passage percolation in two dimensions. In this model, weights are placed on the sites and must be exponentially distributed, one replaces the infimum in the definition of $T$ with supremum, and one considers only oriented paths. Although the model differs from FPP, they are believed to be in the same universality class, and thus these statements should apply to FPP.
\end{enumerate}

\subsection{Progress on bigeodesics and geodesic measures}\label{sec: geodesic_measures_background}

There are few rigorous results toward the bigeodesic conjecture. Although bigeodesics have been completely ruled out in two-dimensional half-planes \cite{WW98}, most other work requires unproven assumptions on the ``limit shape'' for the FPP model. We briefly describe this shape, leaving a more detailed exposition to Section~\ref{sec: tools}. Setting $T(x,y) = T(\lfloor x\rfloor, \lfloor y \rfloor)$ for $x,y \in \mathbb{R}^d$, where, for example, $\lfloor x \rfloor = \sum_{i=1}^d \lfloor x \cdot e_i\rfloor e_i$, one can show that for any fixed $x \in \mathbb{R}^d$, 
\[
g(x) := \lim_{n \to \infty} \frac{T(0,nx)}{n} \text{ exists in probability}
\]
under mild assumptions on the common distribution of $(t_e)$. The resulting function $g: \mathbb{R}^d \to \mathbb{R}$ can be shown to be a deterministic norm, so long as there are not too many low-weight edges. The set
\[
\mathcal{B} = \{x \in \mathbb{R}^d : g(x) \leq 1\}
\]
is called the limit shape for the model. It is easily seen to be convex and compact, with the symmetries of $\mathbb{Z}^d$ that fix the origin. Although $\mathcal{B}$ is expected to be strictly convex and to have uniformly positive curvature, these properties (or even much weaker ones) have not been demonstrated for any i.i.d.~distribution of $(t_e)$. See \cite[Ch.~2]{ADH17} for more details.

In \cite{LN}, Licea-Newman ruled out the existence of bigeodesics with certain given directions. They showed that in dimension $d=2$, for i.i.d.~continuously distributed $(t_e)$, there is a set $D \subset [0,2\pi)$ whose complement has Lebesgue measure zero (this was improved to countable complement by Zerner \cite{Newmanbook}) such that
\begin{equation}\label{eq: good_directions}
\mathbb{P}(\text{there is a bigeodesic with directions } \theta_1,\theta_2) = 0 \text{ for any } \theta_1,\theta_2 \in D.
\end{equation}
Here, a path with vertices $(x_n)$ has direction $\theta$ if $\arg x_n \to \theta$ as $n \to \infty$, and the directions of a bigeodesic refer to the directions of both of its ends. Although this result requires no unproven assumptions, it only gives information if bigeodesics are known to have directions. This latter claim was shown by Newman \cite{N95} under the assumption that the boundary of $\mathcal{B}$ has uniformly positive curvature, and an exponential moment assumption on the i.i.d.~weights $(t_e)$. In \cite{DH17}, Damron-Hanson improved \eqref{eq: good_directions} under an unproven assumption of (local) differentiability of the boundary of $\mathcal{B}$, showing that in dimension $d=2$,
\begin{equation}\label{eq: DH_bigeo}
\mathbb{P}(\text{there is a bigeodesic with an end in direction }\theta) = 0 \text{ for any } \theta \in [0,2\pi).
\end{equation}
A preprint by Ahlberg-Hoffman \cite{AH} states a result that improves the differentiability assumption of \eqref{eq: DH_bigeo}. Both this result and \eqref{eq: good_directions} are only valid for fixed directions, whereas the bigeodesic conjecture is supposed to hold for all $\theta$ simultaneously.

Without unproven assumptions, there is currently only one technique for analyzing bigeodesics, and it is through measures on geodesics. This idea originated in the geodesic graph measures of Damron-Hanson \cite{DH14}. We summarize these measures here and leave the complete construction to Section~\ref{sec: construction}. Let $H$ be any hyperplane through the origin with a normal vector (direction) $\rho$ and define the shifted hyperplane $H_\alpha = H + \alpha \rho$ for $\alpha \in \mathbb{R}$. We would like to take a.s.~limits of geodesics from points to the hyperplanes $H_\alpha$ as $\alpha \to \infty$, but since we do not know these limits exist, we take weak limits. For $x \in \mathbb{Z}^d$, let $\Gamma_x(\alpha)$ be the geodesic from $x$ to $H_\alpha$ (optimizing path for $T(x,H_\alpha)$, defined in the obvious way) and let $\mathbb{G}_\alpha$ be the graph induced by the union of all geodesics in the collection $(\Gamma_x(\alpha) : x \in \mathbb{Z}^d)$. Because the $\Gamma_x(\alpha)$'s coalesce if they touch, every directed path in the graph $\mathbb{G}_\alpha$ is equal to a $\Gamma_x(\alpha)$. This graph is thought of as a random element of a space of directed graphs (the geodesics are directed in the direction they are traveled). Because we want any limit of these graphs to be translation invariant, we average their distributions: for any $n \geq 0$, let $\mathbb{G}_n^\ast$ be the graph obtained on an enlarged space by first sampling an independent uniform random variable $U_n$ on $[0,n]$ and then setting $\mathbb{G}_n^\ast = \mathbb{G}_{U_n}$. These graphs are easily seen to be tight, so we can let
\[
\mathbb{G} = \lim_{k \to \infty} \mathbb{G}^\ast_{n_k} \text{ for some subsequence }(n_k),
\]
where the convergence is in distribution. Any such limit $\mathbb{G}$ is called a geodesic graph. In this language, \cite[Thm.~6.9]{DH14} states:
\begin{theorem}
Let $\mathbb{G}$ be a geodesic graph in dimension $d=2$ constructed in any direction $\rho$. A.s., $\mathbb{G}$ has no doubly-infinite directed paths.
\end{theorem}
Note that any directed path in a geodesic graph $\mathbb{G}$ is a geodesic, so the theorem states that geodesic graphs do not support bigeodesics. The above result holds for all directions, like \eqref{eq: DH_bigeo}, but does not require any limit shape assumptions. It was shown only for $d=2$ because the proof involves first showing that $\mathbb{G}$ has one component a.s., and this requires trapping (planarity) arguments adapted from \cite{LN}.

\subsection{Main results}\label{sec: main_results}

In this section we state the main result of the paper and some of its consequences. We will use the geodesic graph terminology from the previous section; the complete result will appear in Theorem~\ref{thm: main_thm_restated} after the construction of geodesic graphs. Our theorems are proved under two different types of assumptions on the distribution of $(t_e)$:

\begin{enumerate} 
\item[{\bf A:}] The variables $(t_e)_{e \in \mathcal{E}^d}$ are i.i.d., satisfying the moment condition of Cox-Durrett \cite{CD}: if $e_1, \dots, e_{2d}$ are the $2d$ edges incident to the origin, then
\begin{equation}\label{eq: CD_moment}
\mathbb{E}\left[ \min_{i=1, \dots, 2d} t_{e_i} \right]^d < \infty.
\end{equation}
Furthermore, the common distribution of $(t_e)$ is continuous.
\item[{\bf B:}] The variables $(t_e)$ satisfy the conditions of Hoffman \cite{H08}:
\begin{enumerate}
\item their distribution is ergodic with respect to translations of $\mathbb{Z}^d$,
\item their distribution has all the symmetries of $\mathbb{Z}^d$,
\item $\mathbb{E}t_e^{d+\epsilon}<\infty$ for some $\epsilon>0$, and
\item the limit shape $\mathcal{B}$ is bounded.
\end{enumerate}
Furthermore, the distribution has unique passage times: for any two paths $\gamma,\gamma'$ with distinct edge sets, $\mathbb{P}(T(\gamma) = T(\gamma')) = 0$. Last, the distribution has the upward finite energy property from \cite{DH14}: for each $\lambda > 0$ such that $\mathbb{P}(t_e \geq \lambda) > 0$, one has
\begin{equation}\label{eq: upward_finite_energy}
\mathbb{P}(t_e \geq \lambda \mid (t_f)_{f \neq e}) > 0 \text{ a.s.}
\end{equation}
\end{enumerate}

Under these conditions, we have the main result:
\begin{theorem}\label{thm: main_thm}
Assume {\bf A} or {\bf B} and let $\mathbb{G}$ be a geodesic graph in dimension $d \geq 2$ constructed in any direction $\rho$. A.s., $\mathbb{G}$ has no doubly-infinite directed paths.
\end{theorem}
The formal statement appears in Theorem~\ref{thm: main_thm_restated}, after we have given the full construction of geodesic graphs.

As a consequence of Theorem~\ref{thm: main_thm}, we can build subsequential limiting graphs which do not contain doubly-infinite paths. The construction follows \cite[Thm.~1.10]{DH14} (see Section~7.3 there) in a straightforward manner, so we omit the proof.
\begin{corollary}\label{cor: subsequential_limits}
Assume {\bf A} or {\bf B}. Let $H$ be a hyperplane through the origin with normal vector $\rho$ and let $H_\alpha = H+\alpha\rho$ for $\alpha \in \mathbb{R}$. There exists an event $\mathcal{A}$ with $\mathbb{P}(\mathcal{A})=1$ such that for $\omega \in \mathcal{A}$, the following holds. There exists a ($\omega$-dependent) increasing sequence $(\alpha_k)$ of real numbers with $\alpha_k \to \infty$ such that $\mathbb{G}_{\alpha_k}(\omega) \to G(\omega)$, a directed graph with the following properties.
\begin{enumerate}
\item Viewed as an undirected graph, $G$ has no circuits.
\item Each $x \in \mathbb{Z}^d$ has out-degree 1 in $G$.
\item For all $x \in \mathbb{Z}^d$, the set $\{y \in \mathbb{Z}^d : y \to x \text{ in } G\}$ is finite.
\end{enumerate}
\end{corollary}
In the above statements, $y \to x$ means that there is a directed path from $y$ to $x$. Furthermore, a sequence of directed graphs $(G_n)$ converges to a directed graph $G$ (written $G_n \to G$) if for any given $k$, the subgraph of $G_n$ induced by its intersection with $[-k,k]^d$ is equal to the subgraph of $G$ induced by its intersection with $[-k,k]^d$, for all large $n$.

\begin{remark}\label{rem: coalesce}
It is an important question to determine whether geodesic graphs $\mathbb{G}$ in dimension $d \geq 3$ have only one component (coalescence of geodesics). We cannot prove (or disprove) this at the present time, although it is known for $d=2$ (in \cite[Thm.~6.1]{DH14}). The arguments there use planarity in the form of the trapping argument from \cite{LN}.
\end{remark}

\begin{remark}
If the graphs $\mathbb{G}_\alpha$ have an a.s.~limit as $\alpha \to \infty$ (limiting geodesics exist from points to hyperplanes), the methods of this paper can be used to prove that this limit $G$ satisfies the properties of Corollary~\ref{cor: subsequential_limits}. Such limits have been shown to exist in dimension $d=2$ under differentiability assumptions on $\mathcal{B}$ (see \cite{DH17}), and on half-planes in dimension $d=2$ without any unproven assumptions in \cite{ADH15}. Paraphrased, bigeodesics cannot be constructed in a natural way as limits of point-to-hyperplane geodesics.
\end{remark}

\begin{remark}
To prove Theorem~\ref{thm: main_thm}, we use an ``edge-modification'' argument which involves technical issues analogous to those from \cite{DH14}. Because of the added complexity of the proofs in this paper, we have developed an abstract modification procedure in Theorem~\ref{thm: general_modification} in the appendix. We believe that this theorem could be useful for applications in similar settings where one constructs distributional limits of random graphs (for example in the ``metastates'' of \cite{AW, NS} or measures like those of \cite{AH}).
\end{remark}

As stated in Remark~\ref{rem: coalesce}, we do not know how many components are in geodesic graphs in dimensions $d \geq 3$. For this reason, the arguments of \cite{DH14, LN} do not work in our setting. Briefly, the argument of \cite{DH14} for $d=2$ consists of first showing that geodesic graphs have one component (coalescence) and then concluding that any geodesic graph with a doubly-infinite directed path must have several, meeting at so-called encounter points. A well-known result of Burton-Keane \cite{BK89} then rules out existence of encounter points, and therefore doubly-infinite paths. Because the coalescence argument involves identifying a region between two parallel geodesics and modifying edges in this region, we cannot apply it (one cannot define a ``protected'' region in $\mathbb{R}^d$ for $d \geq 3$ whose boundary consists mostly of geodesics that cannot be crossed by other geodesics). We need a different approach, relying on the mass-transport principle, and proving that components that have doubly-infinite paths must intersect certain hyperplanes in bounded sets. (See Lemma~\ref{lem: finite_intersection_component}.) This crucial property allows us to run a modification argument that severs doubly-infinite paths, and the resulting components can be seen to have properties that violate translation invariance. Our overall approach shows that any component of a geodesic graph must intersect these hyperplanes in unbounded sets. In this way, we can understand the trapping argument of \cite{LN}: in two dimensions, if coalescence does not occur, then planarity implies the bounded intersection property, which gives a contradiction.

\subsection{Organization of the paper}
The remainder of the paper is organized as follows. Section~\ref{sec: formal} describes the setting and definition of geodesic measures, along with properties of the different coordinates (Busemann functions and geodesic graphs) on our enlarged space in the construction. In Section~\ref{sec: main_result}, we formally state the main result (Theorem~\ref{thm: main_thm_restated}) for geodesic graphs and give a sketch of the (delicate) arguments involved in the proof. Section~\ref{sec: proof} is devoted to the proof of Theorem~\ref{thm: main_thm_restated}, and begins in Section~\ref{sec: lemmas} with preliminary lemmas about the structure and sizes of components which have doubly-infinite paths (if they were to exist). Section~\ref{sec: absence} covers the full proof of Theorem~\ref{thm: main_thm_restated} in the case that the edge-weights are a.s.~bounded, and in Section~\ref{sec: unbounded_case}, we briefly sketch the modifications needed in the (simpler) case that the edge-weights are unbounded.

\section{Formal definitions and tools for the proof}\label{sec: formal}

\subsection{Setting}\label{sec: tools}

Here we formalize the model and state various standard results that will be used in the proof. First, we take as our probability space $\Omega = [0,\infty)^{\mathcal{E}^d}$ with the product Borel sigma-algebra, and let $(t_e) = (t_e)_{e \in \mathcal{E}^d}$ be the coordinate variables. The passage times $T(x,y)$ are defined as in the introduction for $x,y \in \mathbb{Z}^d$ and are extended to real points $x,y \in \mathbb{R}^d$ as in Section~\ref{sec: geodesic_measures_background}. For $A,B \subset \mathbb{R}^d$, the passage time $T(A,B)$ is defined as the infimum of $T(x,y)$ over all $x \in A$, $y \in B$. The measure $\mathbb{P}$ on $\Omega$ will be such that either {\bf A} or {\bf B} hold for the $t_e$'s. Under either of these assumptions, the shape theorem holds: there exists a deterministic, convex, compact set $\mathcal{B}$ in $\mathbb{R}^d$ such that for each $\epsilon>0$
\[
\mathbb{P}\left( (1-\epsilon) \mathcal{B} \subset \frac{B(t)}{t} \subset (1+\epsilon) \mathcal{B} \text{ for all large }t\right) = 1,
\]
where $B(t) = \{x \in \mathbb{R}^d : T(0,x) \leq t\}$. Furthermore, $\mathcal{B}$ has non-empty interior and has the symmetries of $\mathbb{Z}^d$ that fix the origin. This result was proved in \cite{CD} under assumption {\bf A} and in \cite{Boivin} under {\bf B}.

The set $\mathcal{B}$ is the unit ball of a norm $g$ which satisfies
\[
g(x) = \lim_{n \to \infty} \frac{T(0,nx)}{n} \text{ a.s. and in }L^1 \text{ for fixed } x \in \mathbb{R}^d,
\]
and we can state the shape theorem in the following equivalent form:
\begin{equation}\label{eq: shape_theorem}
\limsup_{\|x\|_1 \to \infty} \frac{|T(0,x)-g(x)|}{\|x\|_1} = 0 \text{ a.s.}
\end{equation}
%We will need the shape theorem which is stated in terms of the limit
%\[
%g(x) = \lim_{n \to \infty} \frac{T(0,nx)}{n}
%\]
%that exists a.s. and in $L^1$ under our assumptions in \eqref{??} for any $x \in \mathbb{R}^d$. One formulation of it is that
%\begin{equation}\label{eq: shape_theorem}
%\limsup_{\|x\| \to \infty} \frac{|T(0,x) - g(x)|}{\|x\|_1} = 0 \text{ a.s..}
%\end{equation}
Note that for $x \in \mathbb{Z}^d$, we have $g(x) \leq \lim_{n \to \infty} T(\Gamma_{nx})/n$, where $\Gamma_{nx}$ is a deterministic path from $0$ to $nx$ with $\|nx\|_1$ many edges. By the law of large numbers, this limit is a.s. equal to $\|x\|_1\mathbb{E}t_e$. Therefore we obtain the simple bound
\begin{equation}\label{eq: g_bound}
\text{for } x \in \mathbb{Z}^d, \text{ one has } g(x) \leq \|x\|_1 \mathbb{E}t_e.
\end{equation}

It is elementary to check (see \cite[Prop.~4.4(a)]{ADH17}) that for any configuration $(t_e)$ in which $\lim_{n \to \infty} T(0,\mathbb{Z}^d \setminus [-n,n]^d) = \infty$, a geodesic exists between any two vertices $x,y$. This limit condition holds a.s.~under {\bf A} or {\bf B}, and this follows immediately from the shape theorem above, since the limit shape $\mathcal{B}$ is bounded. By continuity of the common distribution of $(t_e)$ under {\bf A} or by uniqueness of passage times under {\bf B}, the geodesic between $x$ and $y$ is a.s.~unique. Similar statements can be checked in a straightforward manner for point-to-hyperplane geodesics: given any $x \in \mathbb{Z}^d$ and any hyperplane $H$, there is a unique geodesic from $x$ to $H$ (a path which attains the minimum passage time in the definition of $T(x,H)$).

\subsection{Construction of geodesic measures}\label{sec: construction}

In this section, we construct the geodesic graphs $\mathbb{G}$ and measures on them used in Theorem~\ref{thm: main_thm}. To do this, we fix (for the remainder of the paper) a hyperplane $H$ which is supporting for the limit shape $\mathcal{B}$ at some point $z_0 \in \mathcal{B}$. Recall that this means that $H$ contains $z_0$, and $\mathcal{B}$ intersects exactly one component of $\mathbb{R}^d \setminus H$. 

We next enlarge the probability space $\Omega$ to include edge-weights, passage time differences to hyperplanes, and geodesic graphs. So we set
\[
\Omega_1 = [0,\infty)^{\mathcal{E}^d},~ \Omega_2 = \mathbb{R}^{\mathbb{Z}^d \times \mathbb{Z}^d},~ \text{and } \Omega_3 = \{0,1\}^{\vec{\mathcal{E}}^d},
\]
where $\vec{\mathcal{E}}^d$ is the set of directed edges $\{\langle x,y \rangle : \{x,y \} \in \mathcal{E}^d\}$ of the lattice $\mathbb{Z}^d$. Our full enlarged space will be
\[
\widetilde{\Omega} = \Omega_1 \times \Omega_2 \times \Omega_3,
\]
and the sigma-algebra we use on $\widetilde{\Omega}$ is the product Borel sigma-algebra. 

The measures we use will be pushforwards of our original measure $\mathbb{P}$ through a family of maps. To define these maps, we define the vector $\rho$ to be the unique element of $\mathbb{R}^d$ such that our hyperplane $H$ from above satisfies
\begin{equation}\label{eq: rho_introduction}
H = \{z \in \mathbb{R}^d : z \cdot  \rho = 1\}.
\end{equation}
Then for any $\alpha \in \mathbb{R}$, we set
\begin{equation}\label{eq: hyperplane_definition}
H_\rho(\alpha) = \{z \in \mathbb{R}^d : z \cdot \rho = \alpha\}.
\end{equation}
(We allow this notation $H_\rho(\alpha)$ to apply as above even when $\rho \in \mathbb{R}^d$ is arbitrary.) The map $\Phi_\alpha : \Omega \to \widetilde{\Omega}$ is defined separately for the different coordinates of $\widetilde{\Omega}$. Concerning the first, $\Phi_\alpha$ will map $\omega$ to the edge-weight configuration $(t_e)(\omega)$. For the second, if $\alpha \in \mathbb{R}$, we define the configuration
\[
B_\alpha(\omega) = \left( B_\alpha(x,y)(\omega) : x,y \in \mathbb{Z}^d \right) \in \mathbb{R}^{\mathbb{Z}^d \times \mathbb{Z}^d},
\]
where each entry $B_\alpha(x,y)(\omega)$ is defined as a difference of passage times
\[
B_\alpha(x,y)(\omega)  = T(x, H_\rho(\alpha))(\omega) - T(y,H_\rho(\alpha))(\omega).
\]
These entries will be referred to as Busemann functions or Busemann increments by analogy to Busemann functions for geodesics (see \cite[Ch.~5]{ADH17}). For the third and final coordinate, we define a geodesic graph configuration $\eta_\alpha(\omega)$ as
\[
\eta_\alpha(\omega) = \left( \eta_\alpha(\langle x,y\rangle)(\omega) : \langle x,y\rangle \in \vec{\mathcal{E}}^d\right) \in \{0,1\}^{\vec{\mathcal{E}}^d},
\]
where each entry $\eta_\alpha(\langle x,y \rangle)(\omega)$ is the indicator that the edge $\langle x,y \rangle$ is traversed in a geodesic from some point to $H_\rho(\alpha)$:
\[
\eta_\alpha(\langle x,y \rangle)(\omega) = \begin{cases}
1 &\quad \text{if } \{x,y\} \text{ is in a geodesic from }x\text{ to }H_\rho(\alpha) \\
0 &\quad \text{otherwise}.
\end{cases}
\]
Last, the map $\Phi_\alpha$ is defined as
\[
\Phi_\alpha(\omega) = \left( (t_e)(\omega), B_\alpha(\omega), \eta_\alpha(\omega)\right) \in \widetilde{\Omega},
\]
and the measure $\mu_\alpha$ is the pushforward
\[
\mu_\alpha = \mathbb{P} \circ \Phi_\alpha^{-1}.
\]

Given the list of definitions above, we can take $\alpha \to \infty$ (moving the hyperplane $H_\rho(\alpha)$ to infinity) and formalize the idea of limiting geodesics in the direction of $H$. We would like to define our geodesic graph measure as $\lim_{\alpha \to \infty} \mu_\alpha$, but we do not know that this limit exists. Therefore we need to take a subsequence. To ensure that any subsequential limit is invariant under translations of the lattice, we need to average the measures. Lebesgue measurability of the function $\alpha \mapsto \mu_\alpha(A)$ for fixed events $A$ follows from the argument of \cite[App.~A]{DH14}, and so we can define
\[
\mu_n^* = \frac{1}{n} \int_0^n \mu_\alpha~\text{d}\alpha,
\]
and let $\mu$ be any subsequential limit of the sequence $(\mu_n^*)$. (It is elementary to check that $(\mu_n^*)$ is a tight sequence; this follows from the bound $|B_\alpha(x,y)| \leq T(x,y)$.) 

An important property of $\mu$ is its invariance under translations. Given $z \in \mathbb{Z}^d$, we define the translation by $z$ on the space $\widetilde{\Omega}$ as follows. An arbitrary element of $\widetilde{\Omega}$ we will write as 
\[
\widetilde{\omega} = ((t_e), B,\eta) \in \widetilde{\Omega},
\]
where $B = (B(x,y))_{x,y \in \mathbb{Z}^d}$ and $\eta = (\eta(\langle x,y \rangle))_{\langle x,y \rangle \in \vec{\mathcal{E}}^d}$. Note that under the measure $\mu_\alpha$, the triple $((t_e),B,\eta)$ has the same distribution as $((t_e),B_\alpha,\eta_\alpha)$ does under $\mathbb{P}$. The translation $U_z$ acts on the different coordinates of $\widetilde{\Omega}$ as
\[
U_z((t_e),B,\eta) = ((t_{e+z}), (B(x+z,y+z))_{x,y}, (\eta(\langle x+z,y+z\rangle))_{\langle x,y\rangle}),
\]
where, for example, $e+z$ is the edge $\{w+z,v+z\}$, if $e = \{w,v\}$. Given this definition, we have the following lemma.
\begin{lemma}
For any $z \in \mathbb{Z}^d$, the measure $\mu$ is invariant under $U_z$.
\end{lemma}
\begin{proof}
The proof is identical to \cite[Prop.~3.2]{DH14}, where the statement is proved for an analogous construction of $\mu$ only in two dimensions.
\end{proof}

\subsection{Properties of the Busemann coordinate}

In this section, we state properties of the variables $(B(x,y))$. The proofs of these properties are all identical to those of their two-dimensional analogues in \cite{DH14}; the extension to higher dimensions has been discussed in \cite[Sec.~5.4]{ADH17}.

The first result gives some basic properties of $B(x,y)$. 
%They can be used, for example, to apply the ergodic theorem to sums of the form $\sum_{j=1}^n B((j-1)x,jx)$.
\begin{proposition}
The coordinate $(B(x,y))$ satisfies the following for $x,y,z \in \mathbb{Z}^d$:
\begin{enumerate}
\item $B(x,y) + B(y,z) = B(x,z)~\mu\text{-almost surely},$
\item $B$ is bounded by $T$:
\[
|B(x,y)| \leq T(x,y) ~\mu\text{-almost surely}.
\]
\item The mean of $B$ is equal to
\[
\mathbb{E}_\mu B(x,y) = \rho\cdot (y-x),
\]
where $\rho$ was introduced in \eqref{eq: rho_introduction} and `$\cdot$' is the standard dot product.
\end{enumerate}
\end{proposition}
\begin{proof}
Items 1 and 2 correspond to items 1 and 4 of \cite[Prop.~3.4]{DH14}. Item 3 is given in \cite[Lem.~5.11]{ADH17}.
\end{proof}

Item 3 of the previous proposition gives the mean of $B$. Because $\mu$ is not necessarily ergodic, the large-scale asymptotics of $B(x,y)$ are not necessarily given by $\rho\cdot (y-x)$. Instead, this quantity behaves like $\varrho \cdot (y-x)$ for some random $\varrho$. The following is \cite[Lemma~5.12]{ADH17}.
\begin{theorem}[Shape theorem for $B$]
There exists a nonzero random vector $\varrho\in \mathbb{R}^d$ such that for any $\epsilon>0$,
\[
\mu\left( |B(0,x) - \varrho\cdot x|  > \epsilon \|x\|_1 \text{ for infinitely many } x \in \mathbb{Z}^d\right) = 0.
\]
The vector $\varrho$ satisfies the following conditions:
\begin{enumerate}
\item $\mu$-almost surely, the hyperplane
\[
H_\varrho(1) := \{w \in \mathbb{R}^d : \varrho\cdot w = 1\}
\]
is a supporting hyperplane for $\mathcal{B}$ at $z_0$.
\item The mean of $\varrho$ under $\mu$ is $\rho$:
\[
\mathbb{E}_\mu \varrho = \rho.
\]
\item Consequently, if $\partial \mathcal{B}$ is differentiable at $z_0$, then
\[
\mu(\varrho=\rho) = 1.
\]
\end{enumerate}
\end{theorem}

An important property of $\varrho$ is that it is translation invariant. Indeed, for any $z \in \mathbb{Z}^d$, one has
\begin{equation}\label{eq: TI_varrho}
\varrho(U_z(\widetilde \omega)) = \varrho(\widetilde \omega),~\mu\text{-almost surely}.
\end{equation}
The two-dimensional case of this statement was shown in \cite[Theorem~4.2]{DH14}, and the extension to $d$ dimensions is straightforward.

\subsection{Properties of the graph coordinate}

Given the variables $(\eta(\langle x,y\rangle))$ from the third coordinate of $\widetilde\omega$, we construct a directed graph $\mathbb{G}$ as follows. The vertex set of $\mathbb{G}$ is $\mathbb{Z}^d$. A directed edge $\langle x,y\rangle$ is an edge of $\mathbb{G}$ if and only if $\eta(\langle x,y \rangle) = 1$. We write $x \to z$ for vertices $x,z \in \mathbb{Z}^d$ if there is a directed path in $\mathbb{G}$ from $x$ to $z$.

First we give some basic properties of $\mathbb{G}$. These are stated in \cite[p.~118]{ADH17} for general $d \geq 2$ and are proved for $d=2$ in \cite[Propositions~5.1--5.2]{DH14}. The proofs carry over to the case $d \geq 2$.
\begin{proposition}\label{prop: graph_properties}
With $\mu$-probability one, the following statements hold for $x,y,z \in \mathbb{Z}^d$.
\begin{enumerate}
\item Each directed path in $\mathbb{G}$ is a geodesic. Therefore, viewed as an undirected graph, $\mathbb{G}$ has no circuits.
\item If $x \to y$ in $\mathbb{G}$ then $B(x,y) = T(x,y)$.
\item If $x \to z$ and $y \to z$ in $\mathbb{G}$ then $B(x,y) = T(x,z) - T(y,z)$.
%\item There exists an infinite self-avoiding directed path starting at $x$ in $\mathbb{G}$.
\item In $\mathbb{G}$, $x$ has out-degree one.
%\item Viewed as an undirected graph, $\mathbb{G}$ has no circuits.
\end{enumerate}
\end{proposition}

From items 1 and 4 above, there is exactly one infinite self-avoiding directed path $\Gamma_x$ in $\mathbb{G}$ starting at each $x \in \mathbb{Z}^d$. In the following proposition (which is \cite[Theorem~5.14]{ADH17}), we relate the coordinates $(B(x,y))$ and $(\eta(\langle x,y \rangle))$ to the asymptotic direction of the $\Gamma_x$'s. If $(x_n)$ is a sequence of points in $\mathbb{Z}^d$ and $S \subset \partial \mathcal{B}$, we say that $(x_n)$ is asymptotically directed in $S$ if every subsequential limit of $(x_n/g(x_n))$ is contained in $S$.

\begin{proposition}\label{prop: directedness}
With $\mu$-probability one, for all $x \in \mathbb{Z}^d$, $\Gamma_x$ is asymptotically directed in $S_\varrho = \partial \mathcal{B} \cap H_\varrho(1)$.
\end{proposition}

\subsection{Main result for $\mu$ and outline of the proof}\label{sec: main_result}

Here we state the formal version of Theorem~\ref{thm: main_thm} in terms of the graph $\mathbb{G}$. Given $x \in \mathbb{Z}^d$, we write
\[
C^b_x = \{y \in \mathbb{Z}^d : y \to x\}
\]
for the ``backward cluster'' of $x$ in $\mathbb{G}$. 
\begin{theorem}\label{thm: main_thm_restated}
With $\mu$-probability one, for all $x \in \mathbb{Z}^d$, $C^b_x$ is finite.
\end{theorem}

Before giving the proof in the next section, we give an outline of some of the main ideas. The proof will be by contradiction, so we assume in \eqref{eq: main_assumption_for_contradiction} that $\mu(\#C_0^b = \infty) > 0$. In other words, with positive probability, some components of $\mathbb{G}$ contain doubly-infinite directed paths. First we derive some useful properties of such components. In Lemma~\ref{lem: backbone}, we use a Burton-Keane type argument to prove that a component with a doubly-infinite directed path must a.s.~have exactly one. Specifically, one can define a ``backbone'' $B(C)$ of such a component $C$ with vertex set $\{x \in C : \#C_x^b = \infty\}$ and prove it is a single doubly-infinite directed path. One can then think of the component as consisting of the backbone and then finite directed trees connected to it. 

Due to the existence of a backbone, we derive the bounded intersection property in Lemma~\ref{lem: finite_intersection_component} using the mass transport principle. This property states that for certain hyperplanes $H_{\hat{\varrho}}(n)$ chosen close to $H_{\varrho}(n)$ in Lemma~\ref{lem: hat_rho_definition}, a.s.~the intersection of any component containing a nonempty backbone with these hyperplanes must be bounded. As stated in the introduction, this is the important consequence of existence of bigeodesics that leads to a contradiction.

The most complicated arguments of this paper involve constructing an event $\mathsf{A}_2'$ on which to apply an edge-weight modification. This is an event defined by several conditions (in Section~\ref{sec: modification_argument}) that has the property that when we intersect it with the event that certain edges (in some set $E_N$, which can be taken to be finite due to the bounded intersection property) have large enough weight, then there is a component $C$ of $\mathbb{G}$ that contains no vertices on one side of $H_{\hat{\varrho}}(0)$. In the definition of $\mathsf{A}_2'$, this $C$ corresponds to the component of the point $\xi_N$. We construct $\mathsf{A}_2'$ as a superevent of another event $\mathsf{A}_2$ (defined in Section~\ref{sec: main_construction}) which has the condition that 0 is on a backbone. Because of this condition, the component of $0$ on $\mathsf{A}_2$ has the bounded intersection property, and this allows us to include in the definition of $\mathsf{A}_2'$ the requirement that the component $C$ intersects the strip between $H_{\hat{\varrho}}(0)$ and $H_{\hat{\varrho}}(N)$ in a bounded set. (See condition A2'.4(c).) Pictorially, our weight modification ``severs the backbone'' from the event $\mathsf{A}_2$ to create the component $C$ with the properties listed above.

Once $\mathsf{A}_2'$ has been properly defined and intersected with the event $\{t_e \geq \lambda  \text{ for all } e \in E_N\}$ in \eqref{eq: prefinal}, we can complete the proof using a mass transport argument from \cite{DH14}. For any component, we define its progenitor as one of its vertices which is minimal in a particular ordering of $\mathbb{Z}^d$. Because our component $C$ from $\mathsf{A}_2'$ does not contain vertices on one side of $H_{\hat{\varrho}}(0)$ and has bounded intersection with the strip between $H_{\hat{\varrho}}(0)$ and $H_{\hat{\varrho}}(N)$, it has a unique progenitor. However, a.s.~there are no progenitors: if all vertices of a component send unit mass to the progenitor, then the progenitor receives infinite mass, but sends out only mass one. This violates the mass transport principle and gives the needed contradiction to complete the proof.

\section{Proof of Theorem~\ref{thm: main_thm_restated}}\label{sec: proof}

\subsection{Preliminary lemmas}\label{sec: lemmas}

Here we begin the proof of Theorem~\ref{thm: main_thm_restated}. We will argue by contradiction, so we assume that for some $x \in \mathbb{Z}^d$, $\mu(\#C^b_x = \infty) >0$. By translation invariance, this implies
\begin{equation}\label{eq: main_assumption_for_contradiction}
\mu(\#C^b_0=\infty) > 0.
\end{equation}

First we show that any component of $\mathbb{G}$ with an infinite backward path must have a unique ``backbone.'' From here on, the term component will refer to a component of the undirected version of $\mathbb{G}$ (``weak'' component).
\begin{lemma}\label{lem: backbone}
With $\mu$-probability one, for each component $C$ of $\mathbb{G}$, the subgraph $B(C)$ of $\mathbb{G}$ with vertex set $\{x \in C : \#C^b_x = \infty\}$ is either empty or a doubly-infinite directed path.
\end{lemma}
\begin{proof}
Suppose that for some outcome, $C$ is a component of $\mathbb{G}$ such that $B(C) \neq \emptyset$. We will show that each vertex $x$ in $B(C)$ has in-degree one and out-degree one, and that between each pair of vertices in $B(C)$ there is a directed path. First, each $x$ has out-degree one by Proposition~\ref{prop: graph_properties}. If any $x$ has in-degree larger than one, then let $y,z$ be two vertices of $B(C)$ such that $\langle y,x\rangle$ and $\langle z,x\rangle$ are edges in $\mathbb{G}$. Then $x$ is an ``encounter point'' in the sense of Burton-Keane \cite{BK89}: the component $C$ is infinite and connected, but removal of $x$ from $C$ splits $C$ into at least three infinite connected components. These components respectively contain $\Gamma_x$, $C^b_y$, and $C^b_z$. The method of \cite{BK89} implies that no vertex is an encounter point with positive probability, so a.s.~any such $x$ must have in-degree at most one. Since $\#C^b_x=\infty$, $x$ must have in-degree one.

To show existence of a directed path, let $x,y$ be distinct vertices in $B(C)$. Since $x,y$ are in the same component, there is a vertex self-avoiding path $\pi$ from $x$ to $y$ (in the undirected version of $\mathbb{G}$). Traversing the edges of $\pi$ in order from $x$ to $y$ as $e_1, \dots, e_n$, we find three possible cases.
\begin{enumerate}
\item $e_1$ is oriented toward $x$ in $\mathbb{G}$. In this case, $e_2$ must be oriented toward $e_1$ since otherwise the common endpoint of $e_1$ and $e_2$ would have out-degree at least two, which is impossible. Continuing in this way, we see that $\pi$ is oriented from $y$ to $x$. In other words, $x \in \Gamma_y$.
\item $e_n$ is oriented toward $y$. By a symmetric argument to that in case 1, $\pi$ is oriented from $x$ to $y$, so $y \in \Gamma_x$.
\item $e_1$ and $e_n$ are oriented away from $x$ and $y$ respectively. Traversing $\pi$ in order from $x$ to $y$ as before, there must be a unique vertex $z \in \pi$ such that all the edges of $\pi$ are oriented toward $z$. (To see why, let $e_k$ be the first edge oriented toward $x$ and then apply the argument from case 1 to the following edges to see that they are also oriented toward $x$.) Because $\mathbb{G}$ has no undirected circuits, the infinite backward clusters $C^b_x$ and $C^b_y$ are disjoint, and so $z$ is an encounter point. Again by the Burton-Keane argument, we get a contradiction.
\end{enumerate}
We conclude that either $x \in \Gamma_y$ or $y \in \Gamma_x$, so there is a directed path connecting $x$ and $y$.
\end{proof}

Next, we prove that infinite forward paths $\Gamma_x$ pass from one side to the other of particular hyperplanes only finitely often. In its statement, we use the following modifications of the definition made in \eqref{eq: hyperplane_definition} for $\hat{\varrho} \in \mathbb{R}^d$ and $\alpha \in \mathbb{R}$: 
\[
H_{\hat{\varrho}}^-(\alpha) = \{z \in \mathbb{R}^d : z \cdot \hat{\varrho} < \alpha\},~H_{\hat{\varrho}}^+(\alpha) = \{z \in \mathbb{R}^d : z \cdot \hat{\varrho} > \alpha\}.
\]
%$\#(\Gamma_x \cap H_{\hat{\varrho}}(\alpha))$ is taken to be the number of vertices in $\Gamma_x \cap H_{\hat{\varrho}}(\alpha)$.
\begin{lemma}\label{lem: hat_rho_definition}
$\mu$-almost surely, there exists a $\varrho$-measurable choice of nonzero $\hat{\varrho} \in \mathbb{Q}^d$ such that for all $x \in \mathbb{Z}^d$ and $\alpha \in \mathbb{R}$,
\[
\# (\Gamma_x \cap_{\mathsf{v}} H_{\hat{\varrho}}^-(\alpha)) <\infty,
\]
where the intersection $\cap_{\mathsf{v}}$ is understood as a vertex intersection.
\end{lemma}
\begin{proof}
For a given vector $\varrho$ such that $H_\varrho(1)$ is a supporting hyperplane for $\mathcal{B}$, we first note that $\text{dist}(H_\varrho^-(0), S_\varrho)>0$ (here, dist refers to Euclidean distance, and we recall that $S_\varrho$ was defined in Proposition~\ref{prop: directedness}). Therefore we can choose a rational $\hat{\varrho}$ in a $\varrho$-measurable way so that 
\begin{equation}\label{eq: to_contradict}
\text{dist}(H_{\hat{\varrho}}^-(0), S_\varrho) > 0 ~\mu\text{-a.s.}.
\end{equation}

Now assume for a contradiction that
\begin{equation}\label{eq: for_contradiction_directed}
\mu\left( \text{for some }\alpha \in \mathbb{R} \text{ and } x \in \mathbb{Z}^d,~\#(\Gamma_x \cap_{\mathsf{v}} H_{\hat{\varrho}}^-(\alpha)) = \infty\right)>0.
\end{equation}
On this event, we can choose a (random) sequence $(x_n)$ in $\mathbb{Z}^d$ such that $x_n \in \Gamma_x \cap_{\mathsf{v}} H_{\hat{\varrho}}^-(\alpha)$ for all $n$ and $\|x_n\|_1 \to \infty$ as $n \to \infty$. By compactness, there is a subsequence $(x_{n_k})$ and a point $z \in \partial \mathcal{B}$ such that $x_{n_k}/g(x_{n_k}) \to z$ as $k \to \infty$. By Proposition~\ref{prop: directedness}, $z \in S_\varrho$, $\mu$-a.s.. Then since $x_{n_k} \in H_{\hat{\varrho}}^-(\alpha)$ and $g(x_{n_k}) \to \infty$,
\[
z\cdot \hat{\varrho} = \lim_{k \to \infty} \frac{x_{n_k}}{g(x_{n_k})} \cdot \hat{\varrho} \leq \alpha \lim_{k \to \infty} \frac{1}{g(x_{n_k})} = 0.
\]
Therefore $z$ a.s.~lies in the closure of $H_{\hat{\varrho}}^-(0)$. We conclude that \eqref{eq: for_contradiction_directed} implies that
\[
\mu\left( \text{dist}\left( H_{\hat\varrho}^-(0),S_\varrho\right) = 0\right) > 0,
\]
which contradicts \eqref{eq: to_contradict}. This shows that \eqref{eq: for_contradiction_directed} cannot hold.
%
%
%
%On the other hand, writing $x_{n_k} = w_{n_k} + v_{n_k}$, where $w_{n_k} \in H_{\hat{\varrho}}(0)$ and $v_{n_k} \perp H_{\hat{\varrho}}(0)$, since $\|v_{n_k}\|$ is bounded in $k$,
%\[
%|x_{n_k}/g(x_{n_k}) - w_{n_k}/g(x_{n_k})| = |v_{n_k}/g(x_{n_k})| \to 0 \text{ as } k \to \infty.
%\]
%Since $w_{n_k}/g(x_{n_k}) \in H_{\hat{\varrho}}(0)$ for all $k$, we get $z \in H_{\hat{\varrho}}(0)$. This is a contradiction since $H_{\hat{\varrho}}(0) \cap S_\varrho = \emptyset$.
\end{proof}

For the $\hat{\varrho}$ from Lemma~\ref{lem: hat_rho_definition}, the hyperplane $H_{\hat{\varrho}}(\alpha)$ will contain points of $\mathbb{Z}^d$ for only certain values of $\alpha$. Letting $K$ be an integer such that $K\hat{\varrho}\cdot e_i \in \mathbb{Z}$ for each $i=1, \dots, d$, this set of values of $\alpha$ equals
\begin{align}
&\left\{\alpha : \exists a_1, \dots, a_d \in \mathbb{Z} \text{ such that } \sum_i a_i \hat{\varrho} \cdot e_i = \alpha\right\} \nonumber \\
=~& \frac{1}{K} \left\{ \alpha : \exists a_1, \dots, a_d \in \mathbb{Z} \text{ such that } \sum_i a_i K\hat{\varrho}\cdot e_i = \alpha\right\} \label{eq: frostin_nostin}
\end{align}
The set on the right is the set of all integer linear combinations of the integers $K\hat{\varrho}\cdot e_i$. By a variant of B\'ezout's lemma, if we write $\mathfrak{d}$ for the greatest common divisor of these integers, then \eqref{eq: frostin_nostin} equals $\{n \mathfrak{d}/K : n \in \mathbb{Z}\}$. Therefore if we replace $\hat{\varrho}$ with $K\hat{\varrho}/\mathfrak{d}$, then this set of values of $\alpha$ becomes $\mathbb{Z}$. 
\begin{equation}\label{eq: new_varrho_hat}
\text{From this point on, we use this new value of $\hat{\varrho} \in \mathbb{Z}^d$.}
\end{equation}

We will also need to know that for each component $C$ of $\mathbb{G}$ with a nonempty backbone $B(C)$, $C$ intersects $H_{\hat{\varrho}}(\alpha)$ in a bounded set. 
\begin{lemma}\label{lem: finite_intersection_component}
$\mu$-almost surely, the following holds. If $C$ is a component of $\mathbb{G}$ such that $B(C)$ is nonempty, then for each $n \in \mathbb{Z}$, 
\[
\#(C \cap_{\mathsf{v}} H_{\hat{\varrho}}(n)) < \infty.
\]
As before, $C \cap_{\mathsf{v}} H_{\hat{\varrho}}(n)$ is understood to be a vertex intersection. Therefore, if $B(C)$ is nonempty, then for each $\alpha \in \mathbb{R}$,
\[
C \cap H_{\hat\varrho}(\alpha) \text{ is bounded},
\]
where $\cap$ means intersection in $\mathbb{R}^d$ (viewing $\mathbb{G}$ as a graph embedded in $\mathbb{R}^d$).
\end{lemma}
\begin{proof}
The second statement follows directly from the first. Indeed, each real intersection point $r$ is on an edge with an endpoint in $H_{\hat\varrho}(n)$ for some $n \in \mathbb{Z}$. Furthermore, one has $|n-r|\leq K$ for some universal $K$ depending only on $\hat{\varrho}$, so if $C \cap H_{\hat\varrho}(\alpha)$ is unbounded, we can pick some $n \in \mathbb{Z}$ such that $\#(C\cap H_{\hat\varrho}(n)) = \infty$. 

For the first statement, we will argue by contradiction, so we assume that
\[
\mu(\exists n \in \mathbb{Z},~\text{component }C \text{ with } B(C) \neq \emptyset \text{ such that } \#(C \cap_{\mathsf{v}} H_{\hat{\varrho}}(n)) = \infty) > 0.
\]
%Because $\#(C \cap H_{\hat{\varrho}}(\alpha))$ refers to the size of a vertex set, and $\hat{\varrho}$ is rational, we obtain
%\[
%\mu(\exists \alpha \in \mathbb{Q},~\text{component }C \text{ with } B(C) \neq \emptyset \text{ such that } \#(C \cap H_{\hat{\varrho}}(\alpha)) = \infty) > 0.
%\]
By countable additivity, there is a fixed $n \in \mathbb{Z}$ such that
\[
\mu(\exists ~\text{component }C \text{ with } B(C) \neq \emptyset \text{ such that } \#(C \cap_{\mathsf{v}} H_{\hat{\varrho}}(n)) = \infty) > 0.
\]
Furthermore, there exists $n' \in \mathbb{Z}$ such that
\[
\mu(\exists ~\text{component }C \text{ such that } \#(C \cap H_{\hat{\varrho}}(n)) = \infty \text{ and } \#(B(C) \cap_{\mathsf{v}} H_{\hat{\varrho}}(n')) \geq 1) > 0.
\]

%Suppose now that in some configuration, there is a component $C$ such that $\#(B(C) \cap_{\mathsf{v}} H_{\hat{\varrho}}(n')) \geq 1$. 

We claim that a.s. on the above event, if we follow $B(C)$ in the forward direction (using the orientation given by the fact that it is a directed path), it has a last intersection with $H_{\hat{\varrho}}(n')$. To show this, let $x$ be a vertex in $B(C) \cap_{\mathsf{v}} H_{\hat{\varrho}}(n')$ and let $\pi_x$ be the portion of $B(C)$ starting from $x$ and proceeding indefinitely in the forward direction. Note that $\pi_x = \Gamma_x$. The reason is that these are both paths in $\mathbb{G}$ that begin with $x$ and follow out-edges. Since every vertex in $\mathbb{G}$ has out-degree 1, the paths must be equal. Now we appeal to Lemma~\ref{lem: hat_rho_definition} to deduce that a.s., $\pi_x$ intersects $H_{\hat{\varrho}}(n')$ finitely many times, and so it has a last intersection with this hyperplane. This shows the claim. 

Now define $A$ to be the event that there exists a component $C$ such that $ \#(C \cap_{\mathsf{v}} H_{\hat{\varrho}}(n)) = \infty$, $\#(B(C) \cap_{\mathsf{v}} H_{\hat{\varrho}}(n')) \geq 1$, and $B(C)$ has a last intersection (in the forward direction) with $H_{\hat\varrho}(n')$. By the above arguments, $\mu(A)>0$. We next define a mass-transport function $m : \mathbb{Z}^d \times \mathbb{Z}^d \to \{0,1\}$ as follows. Set
\[
m(x,y) = \begin{cases}
1 & \quad \text{if } y \in x+H_{\hat{\varrho}}(n'-n) \text{ and } y \text{ is the last }\\
&\quad \text{intersection of }B(C(x)) \text{ with } x+ H_{\hat{\varrho}}(n'-n) \\
0 & \quad \text{otherwise},
\end{cases}
\]
where $C(x)$ is the component of $\mathbb{G}$ containing $x$. By the mass transport principle (see \cite{H99} and \cite[Ch.~8]{LPbook}),
\begin{equation}\label{eq: mass_transport}
1\geq \sum_{x \in \mathbb{Z}^d} \mathbb{E}_\mu m(0,x) = \sum_{x \in \mathbb{Z}^d} \mathbb{E}_\mu m(x,0),
\end{equation}
and so $\sum_{x \in \mathbb{Z}^d} m(x,0)$ is finite a.s.. However, on the event $A$, there exists $y \in \mathbb{Z}^d$ such that $\sum_{x \in \mathbb{Z}^d} m(x,y) = \infty$. Indeed, on $A$, choose a suitable component $C$ and an infinite sequence $(x_k)$ in $C \cap H_{\hat{\varrho}}(n)$. Since $H_{\hat{\varrho}}(n') = x_k + H_{\hat{\varrho}}(n'-n)$ for each $k$, the last intersection point $y$ of $B(C)$ with $H_{\hat{\varrho}}(n')$ satisfies $m(x_k,y) = 1$ for all $k$. Since $A$ has positive probability, we get $\sum_{x \in \mathbb{Z}^d} m(x,y) = \infty$ with positive probability for some fixed $y$, and hence also for $y=0$. This is a contradiction, and completes the proof.
\end{proof}

Next is a related result which gives a more explicit bound on the radius of the intersection of a component with a hyperplane. If $C$ is a component with $B(C) \neq \emptyset$, let $x_n(C)$ be the last intersection of $B(C)$ (following $B(C)$ in the forward direction, but now viewing $B(C)$ is a curve in $\mathbb{R}^d$, so that this last intersection is a point in $\mathbb{R}^d$) with $H_{\hat{\varrho}}(n)$. For any $x \in \mathbb{Z}^d$, let
\[
R(x) = \max\left\{ \|z-w\|_1 : z,w \in C(x) \cap H_{\hat{\varrho}}(\hat{\varrho} \cdot x)\right\},
\]
where, as in Lemma~\ref{lem: finite_intersection_component}, $C(x)$ is the component of $\mathbb{G}$ containing $x$ and $\cap$ means intersection in $\mathbb{R}^d$ (viewing $\mathbb{G}$ as a graph embedded in $\mathbb{R}^d$).

\begin{lemma}\label{lem: CUNY_lemma}
For all $n \in \mathbb{Z}$,
\[
\mu(x_n(C(0)) \in \mathbb{Z}^d, R(x_n(C(0))) \geq r, 0 = x_0(C(0))) = \mu(x_{-n}(C(0)) \in \mathbb{Z}^d, R(0) \geq r, 0 = x_0(C(0))).
\]
In particular, by Lemma~\ref{lem: finite_intersection_component},
\[
\lim_{r \to\infty} \sup_n \mu(x_n(C(0)) \in \mathbb{Z}^d,R(x_n(C(0))) \geq r, 0 = x_0(C(0))) = 0.
\]
\end{lemma}
\begin{proof}
We compute, using translation invariance of $\mu$:
\begin{align*}
& \mu(x_n(C(0)) \in \mathbb{Z}^d, R(x_n(C(0))) \geq r, 0 = x_0(C(0))) \\
=~& \sum_{x \in \mathbb{Z}^d} \mu(x = x_n(C(0)), R(x) \geq r, 0 = x_0(C(0))) \\
=~&\sum_{x \in \mathbb{Z}^d} \mu(0 = x_0(C(-x)), R(0) \geq r, -x=x_{-n}(C(-x))) \\
=~& \sum_{x \in \mathbb{Z}^d} \mu(0 = x_0(C(0)), R(0) \geq r, -x = x_{-n}(C(0))) \\
=~& \mu(x_{-n}(C(0)) \in \mathbb{Z}^d, R(0) \geq r, 0 = x_0(C(0))).
\end{align*}
\end{proof}

%then let $y_n(C)$ be the last intersection point of $B(C)$ with $H_{\hat{\varrho}}(n)$ (in $\mathbb{R}^d$), when following $B(C)$ in the forward direction. For any point $y \in H_{\hat{\varrho}}(n)$, let $y^*$ be the closest point in $\mathbb{Z}^d \cap H_{\hat{\varrho}}(n)$ to $y$ (using some deterministic method to break ties). For $x \in \mathbb{Z}^d$,
%\[
%R(x) = \sup_{y \in H_{\hat{\varrho}}(\hat{\varrho} \cdot x) : y^* = x, B(C(y)) \neq \emptyset} \left[ \text{ radius of } C(y) \cap H_{\hat{\varrho}}(\hat{\varrho}\cdot x)\right].
%\]
The next and last lemma states that backbones have asymptotic directions.
\begin{lemma}\label{lem: Arjun}
One has
\[
\mu\left( \lim_{n \to \infty} \frac{y_n(C)}{\|y_n(C)\|_1} \text{ exists for all }C \text{ with } B(C) \neq \emptyset\right) = 1,
\]
where $\dots, y_{-1}(C), y_0(C), y_1(C), \dots$ is any enumeration of the vertices $B(C)$ in forward order. We write $\theta(C)$ for the limit above.
\end{lemma}

%\textcolor{red}{Alternative statement: With $\mu-$probability one, for any component $C$ such that $B(C)\neq \emptyset$, there exist a $\theta=\theta(C)\in S$ such that $B(C)$ has asymptotic direction $\theta$.}

\begin{proof}
The existence of ``asymptotic velocity'' $\theta$ is shown in \cite{CK18} for bi-infinite trajectories in stationary coalescing walk models (of which our graph $\mathbb{G}$ is one). See \cite[Theorem~2.13]{CK18} and the sentence below its statement.
\end{proof}

\subsection{Absence of infinite backward paths}\label{sec: absence}

%\subsubsection{A modification lemma}
%
%{\color{green} maybe have the simple modification lemma from other paper (maybe not needed), and other lemmas for modification.}
In this section, we show that the assumption of existence of backward infinite paths (assumption \eqref{eq: main_assumption_for_contradiction}) leads to a contradiction. Most of the argument will consist of building an event $\mathsf{A}_2$ with many conditions. These conditions will be introduced one by one. We define $S$ as the supremum of the support of the distribution of $t_e$:
\begin{equation}\label{eq: S_def}
S = \sup \{x \geq 0 : \mu(t_e \geq x) > 0\}.
\end{equation}
For this section, we assume that
\begin{equation*}\label{eq: bounded}
S<\infty.
\end{equation*}
The case when $S=\infty$ will be sketched in Section~\ref{sec: unbounded_case}.

\subsubsection{Construction of the main event $\mathsf{A}_2$}\label{sec: main_construction}

We begin by noting that, due to Lemma~\ref{lem: hat_rho_definition}, equation \eqref{eq: new_varrho_hat}, and countable additivity, we can fix a deterministic $\vartheta \in \mathbb{Z}^d$ such that
\begin{equation}\label{eq: pasta_supremo}
\mu(\#C_0^b = \infty, \hat{\varrho} = \vartheta) > 0.
\end{equation}
%{\color{green} We next fix parameters $M', S, \delta$. Then $\lambda$. Then $\epsilon$. Then $N$. Then $M$. We should have
%\begin{equation}\label{eq: N_condition}
%N \geq ??
%\end{equation}
%}
The first step is to show that
\begin{equation}\label{eq: step_1_modification_new}
\mu(\#C_0^b = \infty, \hat{\varrho} = \vartheta, 0 = x_0(C(0)))>0.
\end{equation}
To do this, we give the following claim. Consider an outcome in $\{\#C_0^b = \infty, \hat{\varrho} = \vartheta\}$ and note that for any $k\in \mathbb{N}$, Lemma~\ref{lem: hat_rho_definition} implies that a.s., $x_k(C(0))$ exists as a point of $\mathbb{R}^d$.
\begin{claim}\label{claim: clam}
Either $x_k(C(0)) \in \mathbb{Z}^d$, or at least one endpoint of the edge containing $x_k(C(0))$ is a point of the form $x_n(C(0))$ for $n \in \mathbb{N}$. 
\end{claim}
\begin{proof}
Suppose that $x_k(C(0)) \notin \mathbb{Z}^d$. Let $x$ be the endpoint of the edge of $B(C(0))$ containing $x_k(C(0))$ in its interior such that this edge is directed away from $x$ in $\mathbb{G}$. Note that $x \notin H_{\hat{\varrho}}(k)$, because otherwise the entire edge containing $x_k(C(0))$ would be contained in $H_{\hat{\varrho}}(k)$ and therefore $x_k(C(0))$ (being the last intersection of $B(C(0))$ with $H_{\hat{\varrho}}(k)$) could not be an interior point of the edge. Therefore either $x \in H_{\hat{\varrho}}^+(k)$ or $x \in H_{\hat{\varrho}}^-(k)$. The former case cannot occur because if it did, the edge would cross from $H_{\hat{\varrho}}^+(k)$ to $H_{\hat{\varrho}}^-(k)$, and $x_k(C(0))$ could not be any point of this edge (the backbone would have to cross $H_{\hat{\varrho}}(k)$ again). We conclude that $x \in H_{\hat{\varrho}}^-(k)$, and so $x = x_{x \cdot \hat{\varrho}}(C(0)) \in \mathbb{Z}^d$. 
\end{proof}
Inequality \eqref{eq: step_1_modification_new} follows directly from the the claim after applying a translation. Indeed, the claim implies that a.s., there exists some point in $\mathbb{Z}^d$ that is of the form $x_n(C(0))$. Therefore we can find a deterministic $w \in \mathbb{Z}^d$ and $n \in \mathbb{N}$ such that 
\[
\mu(\#C_0^b = \infty, \hat{\varrho} = \vartheta, w = x_n(C(0)))>0.
\]
Translating by $-w$, we obtain
\[
\mu(\#C_{-w}^b = \infty, \hat{\varrho} = \vartheta, 0 = x_0(C(-w)))>0.
\]
(Here we use that $\hat{\varrho}$ does not change after applying a translation, from \eqref{eq: TI_varrho} and the fact that $\hat{\varrho}$ is $\varrho$-measurable.) This implies \eqref{eq: step_1_modification_new}.

Continuing from \eqref{eq: step_1_modification_new}, we next show that there exist $c_1>0$ and infinitely many $N \in \mathbb{N}$ such that
\begin{equation}\label{eq: step_1_modification}
\mu(\#C_0^b = \infty, \hat{\varrho} = \vartheta,  0 = x_0(C(0)), x_N(C(0)) \in \mathbb{Z}^d) \geq c_1.
\end{equation}
To show this, we observe that Claim~\ref{claim: clam} immediately implies that there exists $C>0$ such that
\begin{equation}\label{eq: step_1_modification_reduction}
\mu\left(\begin{array}{c}
\#C_0^b = \infty, \hat{\varrho} = \vartheta, 0 = x_0(C(0)),\text{ for any }k \in \mathbb{N},~\exists N_k \in \mathbb{N} \\\text{ with } |N_k-k| \leq C \text{ and } x_{N_k}(C(0)) \in \mathbb{Z}^d
\end{array}\right) > 0.
\end{equation}
Therefore, calling this probability $c'$, we obtain for any $k \in \mathbb{N}$,
\begin{align*}
c' &\leq \mu(\#C_0^b = \infty, \hat{\varrho} = \vartheta, 0 = x_0(C(0)),~\exists N_k \in \mathbb{N} \text{ with } |N_k-k| \leq C \text{ and } x_{N_{k}}(C(0)) \in \mathbb{Z}^d) \\
&\leq \sum_{N_{k} : |N_{k}-k| \leq C} \mu(\#C_0^b = \infty, \hat{\varrho} = \vartheta, 0 = x_0(C(0)), x_{N_{k}}(C(0)) \in \mathbb{Z}^d).
\end{align*}
So there exists $N_k$ with $|N_{k}-k| \leq C$ such that
\[
\mu(\#C_0^b = \infty, \hat{\varrho} = \vartheta, 0 = x_0(C(0)), x_{N_{k}}(C(0)) \in \mathbb{Z}^d) \geq \frac{c'}{2C+1}.
\]
Since this holds for all $k \in \mathbb{N}$, it implies \eqref{eq: step_1_modification} with $c_1 = c'/(2C+1)$.

Now that we have established \eqref{eq: step_1_modification}, we need to add more conditions into its associated event. For $M',N \in \mathbb{N}$, $\delta,\epsilon>0$, and $y \in \mathbb{Z}^d \cap H_{\hat{\varrho}}(0)$ with $\|y\|_1 \leq M'$, define $\mathsf{A}_1 = \mathsf{A}_1(M',N,y,\delta,\epsilon)$ as the event that the following occur:
\begin{enumerate}
\item[A1.1] $\#C_0^b = \infty$, $\hat{\varrho} = \vartheta$, $0 = x_0(C(0))$, and $x_N(C(0)) \in \mathbb{Z}^d$;
\item[A1.2] $\Gamma_y$ contains a vertex within $\ell^1$-distance $\epsilon \|x_N(C(0))\|_1$ of $x_N(C(0))$, but $\Gamma_y \cap \Gamma_0 = \emptyset$; and
\item[A1.3] for all vertices $v \in \Gamma_y$ with $\|v-y\|_1 \geq M'$, one has
\[
T(y,v) \leq \|v-y\|_1 (S-\delta).
\]
\end{enumerate}
Here, $S$ is the supremum defined in \eqref{eq: S_def}. We will prove that for some fixed choice of $\delta$, the following holds: for any $\epsilon>0$, there exist $y,M',$ and also $c_2>0$ such that for infinitely many $N$,
\begin{equation}\label{eq: A_1_bound}
\mu(\mathsf{A}_1) \geq c_2.
\end{equation}
See Figure~\ref{fig: A_1} for an illustration of the event $\mathsf{A}_1$.

%\begin{center}
	\begin{figure}
\hbox{\hspace{1.3cm}\includegraphics[width=0.8\textwidth, trim={0 8cm 3cm 3cm}, clip]{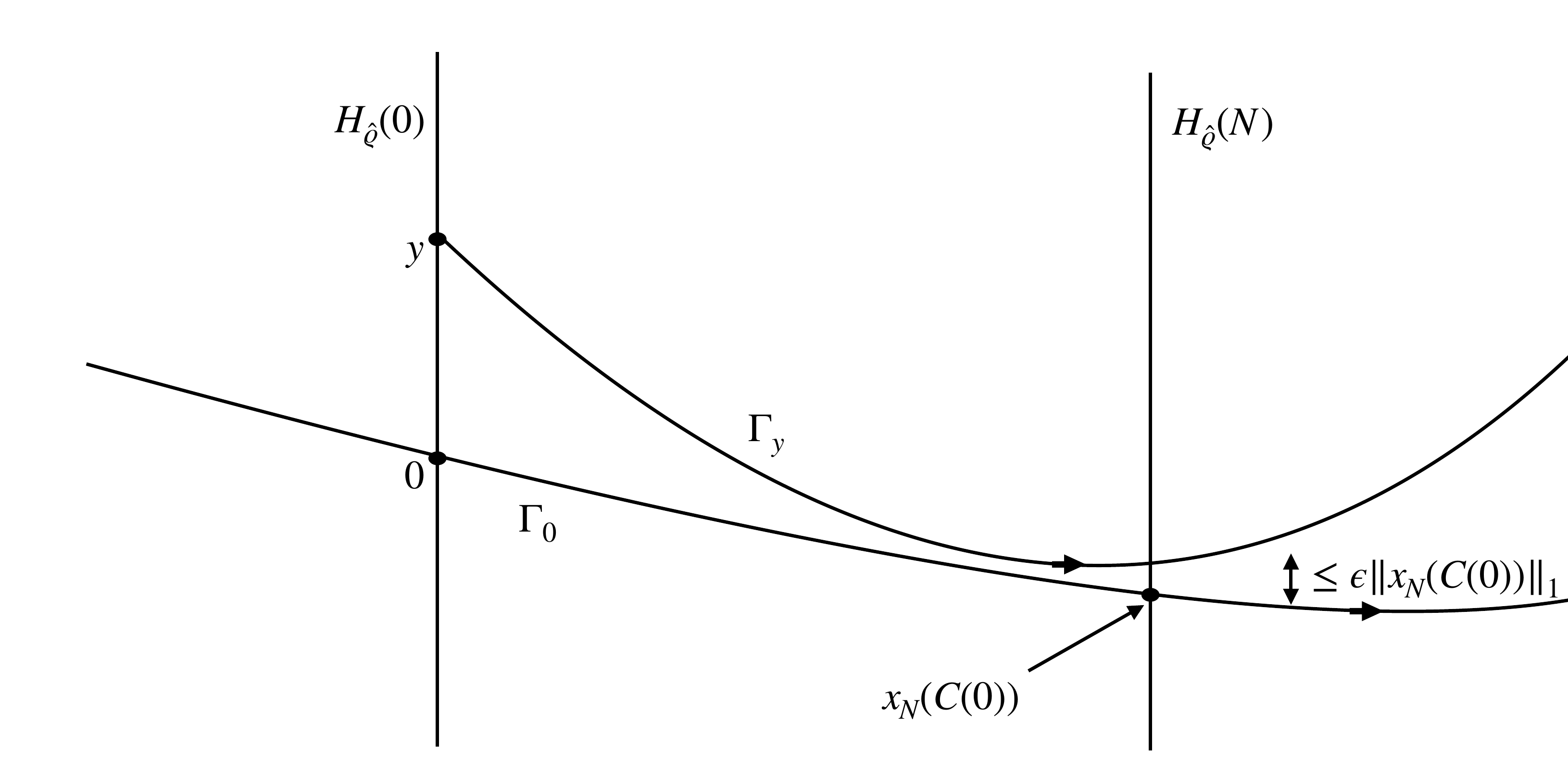}}
  \caption{Depiction of the event $\mathsf{A}_1$. The origin is on the backbone of its own component and is the last intersection point of this backbone with $H_{\hat{\varrho}}(0)$. The geodesic $\Gamma_y$ starting from $y$ comes within distance $\epsilon\|x_N(C(0))\|_1$ of $x_N(C(0))$, but does not intersect $\Gamma_0$. Arrows on the paths indicate their directions.}
  \label{fig: A_1}
\end{figure}
%\end{center}

To show \eqref{eq: A_1_bound}, recall Lemma~\ref{lem: Arjun}, which implies that on the event in \eqref{eq: step_1_modification}, the direction $\Theta(C(0))$ exists a.s.. Therefore given $\epsilon>0$, we can pick a deterministic $D \subset \mathbb{R}^d$ with $\ell^1$-diameter at most $\epsilon/2$ and $c_3>0$ such that for infinitely many $N \in \mathbb{N}$,
\begin{equation}\label{eq: quest_bar}
\mu(\#C_0^b = \infty, \hat\varrho = \vartheta, 0 = x_0(C(0)), x_N(C(0)) \in \mathbb{Z}^d, \Theta(C(0)) \in D)\geq c_3.
\end{equation}
(To do this, cover the unit $\ell^1$-ball with finitely many balls of diameter $\epsilon/2$ and choose $D = D_N$ to be so that the probability is at least $c_3$. Then take a subsequence of $N$ such that $D_N$ is constant.) For $y \in \mathbb{Z}^d \cap H_{\hat{\varrho}}(0)$, define $E_y$ as the event that $\#C_y^b = \infty$, $y = x_0(C(y))$, and $\Theta(C(y)) \in D$. Note that $E_0$ contains the event in \eqref{eq: quest_bar} (call it $E_0'$). By translation invariance, $\mu(E_y \text{ occurs for infinitely many } y \in \mathbb{Z}^d \cap H_{\hat{\varrho}}(0) \mid E_0) = 1$, and so $\mu(E_y \text{ occurs for infinitely many }y \in \mathbb{Z}^d \cap H_{\hat{\varrho}}(0) \mid E_0') = 1$. Therefore, we can find $R>0$ such that $\mu(\cup_{0 < \|y\|_1\leq R} E_y \mid E_0') > 1/2$, where the union is over $y \in \mathbb{Z}^d \cap H_{\hat{\varrho}}(0)$. This implies that for infinitely many $N \in \mathbb{N}$,
\[
\mu(\#C_0^b = \infty, \hat\varrho= \vartheta, 0 = x_0(C(0)), x_N(C(0)) \in \mathbb{Z}^d, \Theta(C(0)) \in D, \cup_{0 < \|y\|_1 \leq R}E_y) \geq c_3/2.
\]
By a union bound, there is $c_4>0$ such that for infinitely many $N \in \mathbb{N}$, there exists $y_N \in \mathbb{Z}^d \cap H_{\hat{\varrho}}(0)$ with $0 < \|y_N\|_1 \leq R$ and
\[
\mu(\#C_0^b = \infty, \hat\varrho= \vartheta, 0 = x_0(C(0)), x_N(C(0)) \in \mathbb{Z}^d, \Theta(C(0)) \in D, E_{y_N}) \geq c_4.
\]
By restricting to a subsequence of values of $N$, there is one $y \in \mathbb{Z}^d \cap H_{\hat{\varrho}}(0)$ with $0 < \|y\|_1 \leq R$ such that
\begin{equation}\label{eq: lasagna_supremo}
\mu(\#C_0^b = \infty, \hat\varrho= \vartheta, 0 = x_0(C(0)), x_N(C(0)) \in \mathbb{Z}^d, \Theta(C(0)) \in D, E_y) \geq c_4.
\end{equation}

Define $F = \{\#C_0^b = \infty, \#C_y^b = \infty, \Theta(C(0)) \in D, \Theta(C(y)) \in D\}$, which is a superevent of the event in \eqref{eq: lasagna_supremo}, and note that for any outcome in $F$, we have $\|\Theta(C(0))- \Theta(C(y))\|_1 \leq \epsilon/2$. So for such outcomes, we can choose a (random) $n_0$ such that 
\begin{equation}\label{eq: n_0_equation}
\left\| \frac{y_n(C(0))}{\|y_n(C(0))\|_1} - \frac{y_m(C(y))}{\|y_m(C(y))\|_1}\right\|_1 \leq \epsilon \text{ for } m,n \geq n_0.
\end{equation}
(Here we have enumerated the vertices so that $y_0(C(0)) = 0$ and $y_0(C(y)) = y$.) Last, pick a deterministic $n_0$ such that
\begin{equation}\label{eq: n_0_again}
\mu(\eqref{eq: n_0_equation} \text{ holds } \mid F) \geq 1-c_4/2.
\end{equation}
Then if $A_N$ is the event in \eqref{eq: lasagna_supremo} and $G$ is the event in \eqref{eq: n_0_equation} (with our deterministic $n_0$),
\[
\mu(A_N \cap G) = \mu(A_N \cap G \cap F) = \mu(A_N \cap G \mid F) \mu(F).
\]
Since $\mu(A_N \mid F) \geq \mu(A_N \cap F) = \mu(A_N) \geq c_4$, we obtain from \eqref{eq: n_0_again} that $\mu(A_N \cap G \mid F) \geq c_4/2$. Therefore (for infinitely many $N \in \mathbb{N}$)
\begin{equation}\label{eq: 1_and_2}
\mu(A_N \cap G) \geq (c_4/2)\mu(F) \geq c_4^2/2.
\end{equation}

We will now argue that for $N$ large enough and for any outcome in $A_N \cap G$, condition A1.2 holds. Take such an outcome and note that on this outcome, $y = x_0(C(y))$ and $0 = x_0(C(0))$. However, $y \in B(C(y))$ and $0 \in B(C(0))$, so since $y \neq 0$, $B(C(y)) \neq B(C(0))$ (and in fact they are disjoint). Since $\Gamma_0 \subset B(C(0))$ and $\Gamma_y \subset B(C(y))$, we get $\Gamma_y \cap \Gamma_0 =\emptyset$. Next, for $N$ large, the index $n$ such that $y_n(C(0)) = x_N(C(0))$ satisfies $n \geq n_0$. Furthermore, picking any $m$ such that $\|y_m(C(y))\|_1 = \|x_N(C(0))\|_1$, we obtain for $N$ large that $m \geq n_0$ as well. Since $G$ occurs, we obtain
\[
\left\| \frac{x_N(C(0))}{\|x_N(C(0))\|_1} - \frac{y_m(C(y)))}{\|y_m(C(y))\|_1}\right\|_1 \leq \epsilon.
\]
Therefore $\|x_N(C(0)) - y_m(C(y))\|_1 \leq \epsilon \|x_N(C(0))\|_1$. This completes the proof that $A_N \cap G$ implies condition A1.2.

Last, we need to add condition A1.3 to our events, and it follows directly from the shape theorem. Under our assumption that $S$, defined in \eqref{eq: S_def}, is finite, we can pick $\delta>0$ such that $\mathbb{E}t_e \leq S-2\delta$. (This follows under {\bf A} because the distribution of $t_e$ is continuous and under {\bf B} because if $t_e  =S$ with probability one, then uniqueness of passage times fails.) Combining this with \eqref{eq: g_bound}, we obtain $g(x) \leq (S-2\delta)\|x\|_1$ for all $x \in \mathbb{Z}^d$. By the shape theorem in \eqref{eq: shape_theorem},
\[
T(0,x) \leq \|x\|_1(S - \delta) \text{ for all } x\in \mathbb{Z}^d \text{ with } \|x\|_1 \text{ sufficiently large}
\]
with probability one. Now we choose $M'$ such that $\|y\|_1 \leq M'$ and
\[
T(y,v) \leq (S-\delta)\|y-v\|_1 \text{ for all } v \in \mathbb{Z}^d \text{ with } \|v-y\|_1 \geq M'
\]
has probability $>1-c_4^2/4$. (This is the condition A1.3.) Intersecting this event with the event in \eqref{eq: 1_and_2}, we obtain \eqref{eq: A_1_bound} for $c_2=c_4^2/4$.

To add the last conditions into our events, we finally define $\mathsf{A}_2 = \mathsf{A}_2(M,M',N,y,\delta,\epsilon)$ (for $M',N,y,\delta,\epsilon$ as in the definition of $\mathsf{A}_1$, but also with an integer $M >0$) as the event that the following occur:
\begin{enumerate}
\item $\mathsf{A}_1$, and
\item $\Gamma_z \cap \Gamma_0 = \emptyset$ for any $z \in \mathbb{Z}^d$ which is an endpoint of an edge that contains a point $w \in \mathbb{R}^d$ with
\begin{enumerate}
\item $w \in H_{\hat{\varrho}}(0)$ and $\|w\|_1 \geq M'$, or
\item $w \in H_{\hat{\varrho}}(N)$ and $\|w-x_N(C(0))\|_1 \geq M'$, or
\item $0 \leq w\cdot \hat{\varrho} \leq N$ and $w$ has Euclidean distance $\geq M$ from the line through $0$ and $\hat{\varrho}$.
\end{enumerate}
\end{enumerate}
See Figure~\ref{fig: A_2} for an illustration of the event $\mathsf{A}_2$.

	\begin{figure}
\hbox{\hspace{1.3cm}\includegraphics[width=0.8\textwidth, trim={0 7cm 3cm 3cm}, clip]{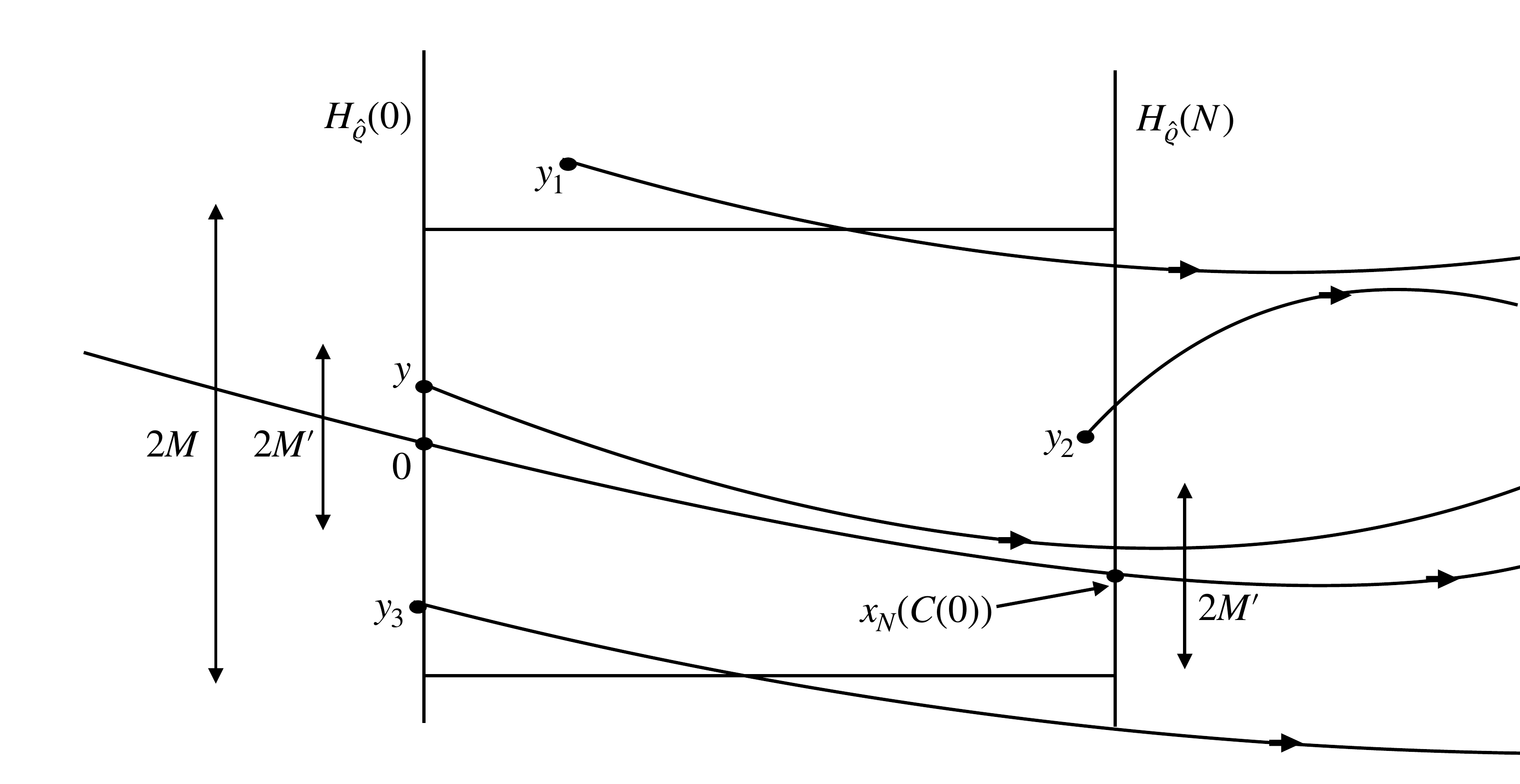}}
  \caption{Depiction of the event $\mathsf{A}_2$. In addition to the conditions of $\mathsf{A}_1$, we also impose items 2(a-c). In the figure, $y$ is within distance $M'$ of the origin. The point $y_3 \in H_{\hat{\varrho}}(0)$ is not, so the path coming from it does not intersect $\Gamma_0$. Similarly, the path coming from $y_2$ intersects $H_{\hat{\varrho}}(N)$ at a distance greater than $M'$ from $x_N(C(0))$, so it does not touch $\Gamma_0$. Last, the path coming from $y_1$ starts outside the set described in 2(c), so it does not intersect $\Gamma_0$ either.}
  \label{fig: A_2}
\end{figure}

We will now show that for some fixed choice of $\delta$, the following holds: for any $\epsilon>0$, there exist $y,M'$ such that for infinitely many $N$ and some choice of $M = M(N)$,
\begin{equation}\label{eq: last_event}
\mu(\mathsf{A}_2) > 0.
\end{equation}
To do this, first fix $\delta$ such that for any $\epsilon>0$, there exist $y,M'$, and $c_2>0$ such that for infinitely many $N$, \eqref{eq: A_1_bound} holds. Next, applying Lemma~\ref{lem: CUNY_lemma}, we can pick $r_0>0$ such that
\begin{equation}\label{eq: quest_supremo}
\sup_n \mu(x_n(C(0)) \in \mathbb{Z}^d, R(x_n(C(0))) \geq r_0, 0 = x_0(C(0))) \leq \frac{c_2}{4}.
\end{equation}
Now given any $\epsilon>0$, we replace $M'$ with $\max\{M',r_0\}$, and note that \eqref{eq: A_1_bound} still holds. We obtain
\[
\text{ for this }M' \text{ and infinitely many } N,~\mu(\mathsf{A}_1, \text{ 2(a), 2(b)}) \geq \frac{c_2}{2},
\]
by intersecting $\mathsf{A}_1$ with the complement of the event in \eqref{eq: quest_supremo} for $n=0$ and $N$. To obtain a lower bound for the probability of this event intersected with 2(c), we simply apply the second statement of Lemma~\ref{lem: finite_intersection_component} over $\alpha \in [0,N]$ for a fixed $N$ to obtain
\[
\text{ for this }M' \text{ and infinitely many } N,~\mu(\mathsf{A}_1, \text{ 2(a), 2(b), 2(c)}) \geq \frac{c_2}{4}
\]
for a large value of $M=M(N)$. This proves \eqref{eq: last_event}.

\subsubsection{Modification argument}\label{sec: modification_argument}
In this section, we modify the weights for edges in a particular subset of the lattice. These edges will be chosen so that after the modification, the backbone of $C(0)$ is severed, and this will create a minimal element in $C(0)$ in a certain lexicographic ordering. In the next section, we will show that no such minimal element can exist, and this will contradict the assumption of existence of the backbone.

To begin, we recall all of the conditions of the event $\mathsf{A}_2$. For a given $\epsilon>0$, $y \in \mathbb{Z}^d \cap H_{\hat{\varrho}}(0)$, integer $M'$ with $\|y\|_1\leq M'$, a number $\delta>0$, and two integers $M,N$,
\begin{enumerate}
\item[A2.1] $\#C_0^b=\infty$, $\hat{\varrho}= \vartheta$, $0 = x_0(C(0))$, and $x_N(C(0)) \in \mathbb{Z}^d$;
\item[A2.2] $\Gamma_y$ contains a vertex within $\ell^1$-distance $\epsilon \|x_N(C(0))\|_1$ of $x_N(C(0))$, but $\Gamma_y \cap \Gamma_0 = \emptyset$; 
\item[A2.3] for all vertices $v \in \Gamma_y$ with $\|v-y\|_1 \geq M'$, one has
\[
T(y,v) \leq \|v-y\|_1 (S-\delta);
\]
\item[A2.4] $\Gamma_z \cap \Gamma_0 = \emptyset$ for any $z \in \mathbb{Z}^d$ which is an endpoint of an edge that contains a point $w \in \mathbb{R}^d$ with
\begin{enumerate}
\item $w \in H_{\hat{\varrho}}(0)$ and $\|w\|_1 \geq M'$, or
\item $w \in H_{\hat{\varrho}}(N)$ and $\|w-x_N(C(0))\|_1 \geq M'$, or
\item $0 \leq w\cdot \hat{\varrho} \leq N$ and $w$ has Euclidean distance $\geq M$ from the line through $0$ and $\hat{\varrho}$.
\end{enumerate}
\end{enumerate}
For any $\epsilon>0$, we fix $y,M'$ such that \eqref{eq: last_event} occurs for infinitely many $N$ and $M = M(N)$ (for the fixed choice of $\delta>0$ there).

Unfortunately we cannot apply the modification lemma directly to the event $\mathsf{A}_2$ because many of its conditions reference $C(0)$, and because $C(0)$ will change after modifying edge-weights, they will not be stable under the modification. Therefore we define a superevent $\mathsf{A}_2'$, which depends on an additional vertex $\xi_N \in \mathbb{Z}^d \cap H_{\vartheta}(N)$, as follows:
\begin{enumerate}
\item[A2'.1] the only $w \in \mathbb{R}^d$ contained in $\Gamma_{\xi_N}$ with $w \cdot \vartheta \in [0,N]$ is $w=\xi_N$;
\item[A2'.2] $\Gamma_y$ contains a vertex within $\ell^1$-distance $\epsilon \|\xi_N\|_1$ of $\xi_N$, but $\Gamma_y \cap \Gamma_{\xi_N} = \emptyset$; 
\item[A2'.3] for all vertices $v \in \Gamma_y$ within $\ell^1$-distance $\epsilon \|\xi_N\|_1$ of $\xi_N$ and with $\|v-y\|_1 \geq M'$, one has
\[
T(y,v) \leq \|v-y\|_1 (S-\delta);
\]
\item[A2'.4] $\Gamma_z \cap \Gamma_{\xi_N} = \emptyset$ for any $z \in \mathbb{Z}^d$ which is an endpoint of an edge that contains a point $w \in \mathbb{R}^d$ with
\begin{enumerate}
\item $w \in H_{\vartheta}(0)$ and $\|w\|_1 \geq M'$, or
\item $w \in H_{\vartheta}(N)$ and $\|w-\xi_N\|_1 \geq M'$, or
\item $0 \leq w\cdot \vartheta \leq N$ and $w$ has Euclidean distance $\geq M$ from the line through $0$ and $\vartheta$.
\end{enumerate}
\end{enumerate}
See Figure~\ref{fig: A_2_prime} for an illustration of the event $\mathsf{A}_2'$.
%%%TO REVERT IF MESS UP
%\begin{enumerate}
%\item $\hat{\varrho}= \vartheta$ and the only vertex $w$ of $\Gamma_{\xi_N}$ with $w \cdot \hat\varrho \in [0,N]$ is $w=\xi_N$;
%\item $\Gamma_y$ comes within $\ell^1$ distance $\epsilon \|\xi_N\|_1$ of $\xi_N$, but $\Gamma_y \cap \Gamma_{\xi_N} = \emptyset$; 
%\item for all vertices $v \in \Gamma_y$ with $\|v-y\|_1 \geq M'$, one has
%\[
%T(y,v) \leq \|v-y\|_1 (S-\delta);
%\]
%\item $\Gamma_z \cap \Gamma_{\xi_N} = \emptyset$ for any $z \in \mathbb{Z}^d$ satisfying any of these conditions:
%\begin{enumerate}
%\item $z \in H_{\hat{\varrho}}(0)$ and $\|z\|_1 \geq M'$,
%\item $z \in H_{\hat{\varrho}}(N)$ and $\|z-\xi_N\|_1 \geq M'$, or
%\item $0 \leq z\cdot \hat{\varrho} \leq N$ and $z$ has Euclidean distance $\geq M$ from the line through $0$ and $\hat{\varrho}$.
%\end{enumerate}
%\end{enumerate}

	\begin{figure}
\hbox{\hspace{1.3cm}\includegraphics[width=0.8\textwidth, trim={0 8cm 3cm 3cm}, clip]{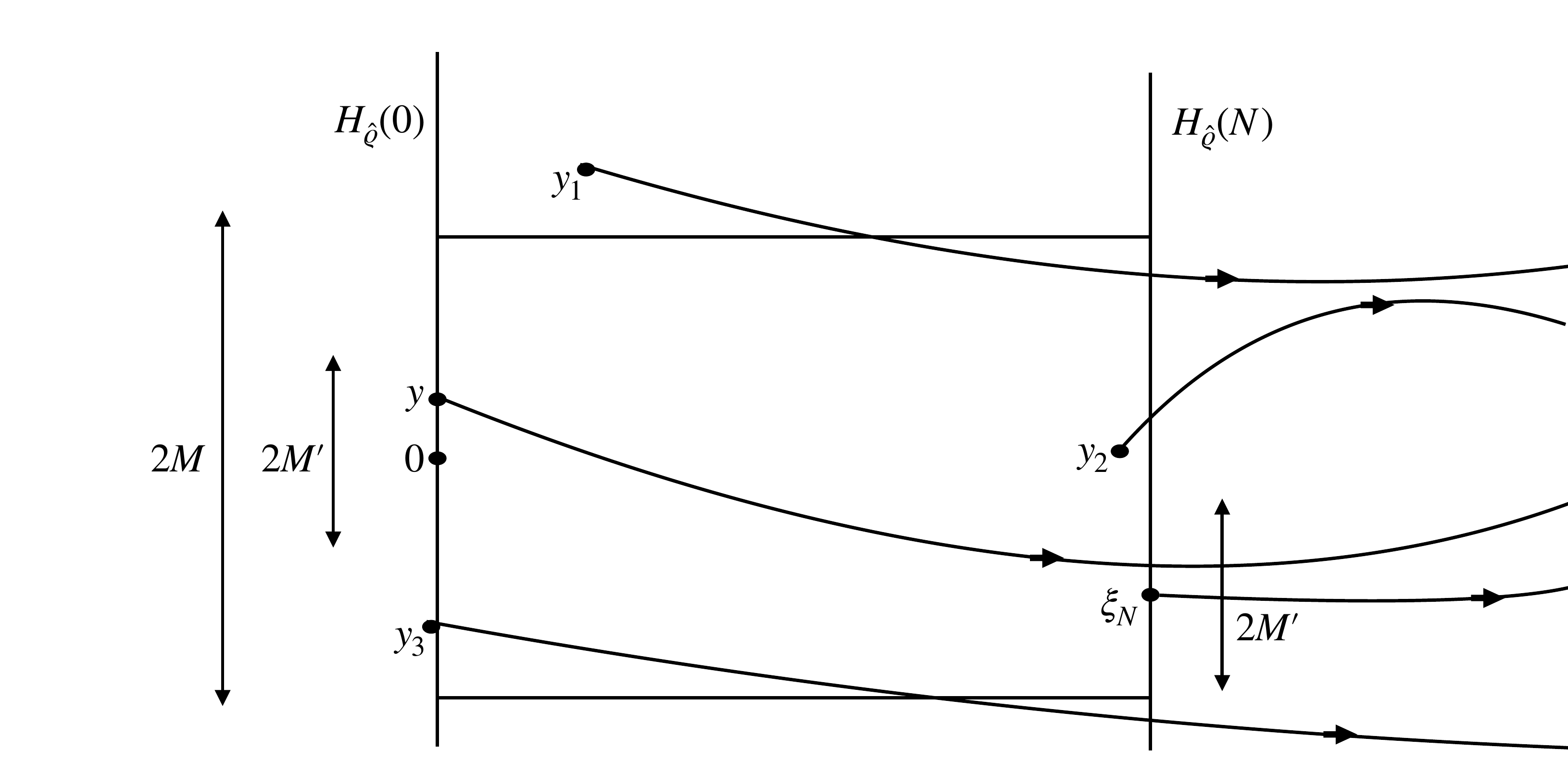}}
  \caption{Depiction of the event $\mathsf{A}_2'$. This figure is the same as Figure~\ref{fig: A_2}, except now there is no reference to the path coming from 0. Instead, what was the point $x_N(C(0))$ is now labeled as $\xi_N$, and the requirement is that $\Gamma_{\xi_N}$ does not re-enter the region to the left of $H_{\hat{\varrho}}(N)$. The paths through the points $y_i$ do not intersect $\Gamma_{\xi_N}$, much like they did not intersect $\Gamma_0$ before.}
  \label{fig: A_2_prime}
\end{figure}

Because $\mathsf{A}_2$ implies $\cup_{\xi_N} \mathsf{A}_2'$, where the union is over $\xi_N \in \mathbb{Z}^d \cap H_{\vartheta}(N)$, a union bound shows that for any $\epsilon>0$, there exist $y,M',\xi_N$ with $\|y\|_1 \leq M'$ such that
\begin{equation}\label{eq: last_event_2}
\mu(\mathsf{A}_2') > 0
\end{equation}
for infinitely many $N$ and $M=M(N)$.

Last, we must define the set of edges whose weights we will increase. We would like to keep all geodesics starting from either $y$ or points $z$ satisfying any of A2'.4(a)-(c). Furthermore, we can only modify weights for a finite set of edges. Therefore we define $\Xi = \Xi(M,N)$ as the set of edges $e$ with both endpoints in 
\[
S(M,N) = \{z : 0 \leq z \cdot \vartheta \leq N, ~z \text{ has Euclidean distance } \leq M \text{ from the line through }0 \text{ and } \vartheta\}
\]
such that both
\begin{enumerate}
\item $e$ is not in $\Gamma_y$ and
\item $e$ is not in $\Gamma_z$ for any $z$ satisfying any of the conditions A2'.4(a)-(c).
\end{enumerate}
%{\color{green} The following will need to be changed (and a new event implied by $\mathbf{A}_2$ will need to be defined) so that everything is increasing. In particular, we need (a) deal with $\hat{\varrho} = \vartheta$ being Busemann-measurable, (b) $T(y,v)$ should be $T_{\Gamma_y}(y,v)$ (or something), (c) all these conditions seem to assume we are in the support of $\mu$.}
%
%Let $\Xi = \Xi(M,N)$ be the set of edges $e$ with both endpoints in the set $S(M,N) = \{z : 0 \leq z \cdot \hat{\varrho} \leq N, z \text{ has Euclidean distance $\leq M$ from the line through $0$ and $\hat{\varrho}$}\}$ such that either $e \notin \Gamma_x$ for all $x \in \mathbb{Z}^d$ or $e \in \Gamma_0$. 
For $N$ as in \eqref{eq: last_event_2}, we choose a deterministic set of edges $E_N$ with endpoints in $S(M,N)$ such that
\[
\mu(\mathsf{A}_2', \Xi = E_N) > 0.
\]
In the appendix, we will show that for
\begin{equation}\label{eq: lambda_choice}
\lambda = S - \frac{\delta}{2},
\end{equation}
one has
\begin{equation}\label{eq: prefinal}
\mu(\mathsf{A}_2', \Xi = E_N, t_e \geq \lambda \text{ for all } e \in E_N) > 0
\end{equation}
for infinitely many $N$ and $M=M(N)$. This will be proved as Corollary~\ref{cor: our_events_modify}, and will follow from an abstract edge modification argument stated as Theorem~\ref{thm: general_modification} in Appendix~\ref{sec: appendix}.

\subsubsection{Completing the proof}\label{sec: completing}

In this final section, we show that \eqref{eq: prefinal} leads to a contradiction. To this end, we consider an outcome in the event $\mathsf{A}_2' \cap \{\Xi = E_N, t_e \geq \lambda \text{ for all } e \in E_N\}$ and first prove that
\begin{equation}\label{eq: last_claim_1}
\text{for all } z \in \mathbb{Z}^d \text{ with } z \cdot \vartheta \leq 0,~\Gamma_z \cap \Gamma_{\xi_N} = \emptyset.
\end{equation}
In other words, the edge modification we did in the last section removed the backbone of the cluster of $\xi_N$.

	\begin{figure}
\hbox{\hspace{1.3cm}\includegraphics[width=0.8\textwidth, trim={0 8cm 3cm 3cm}, clip]{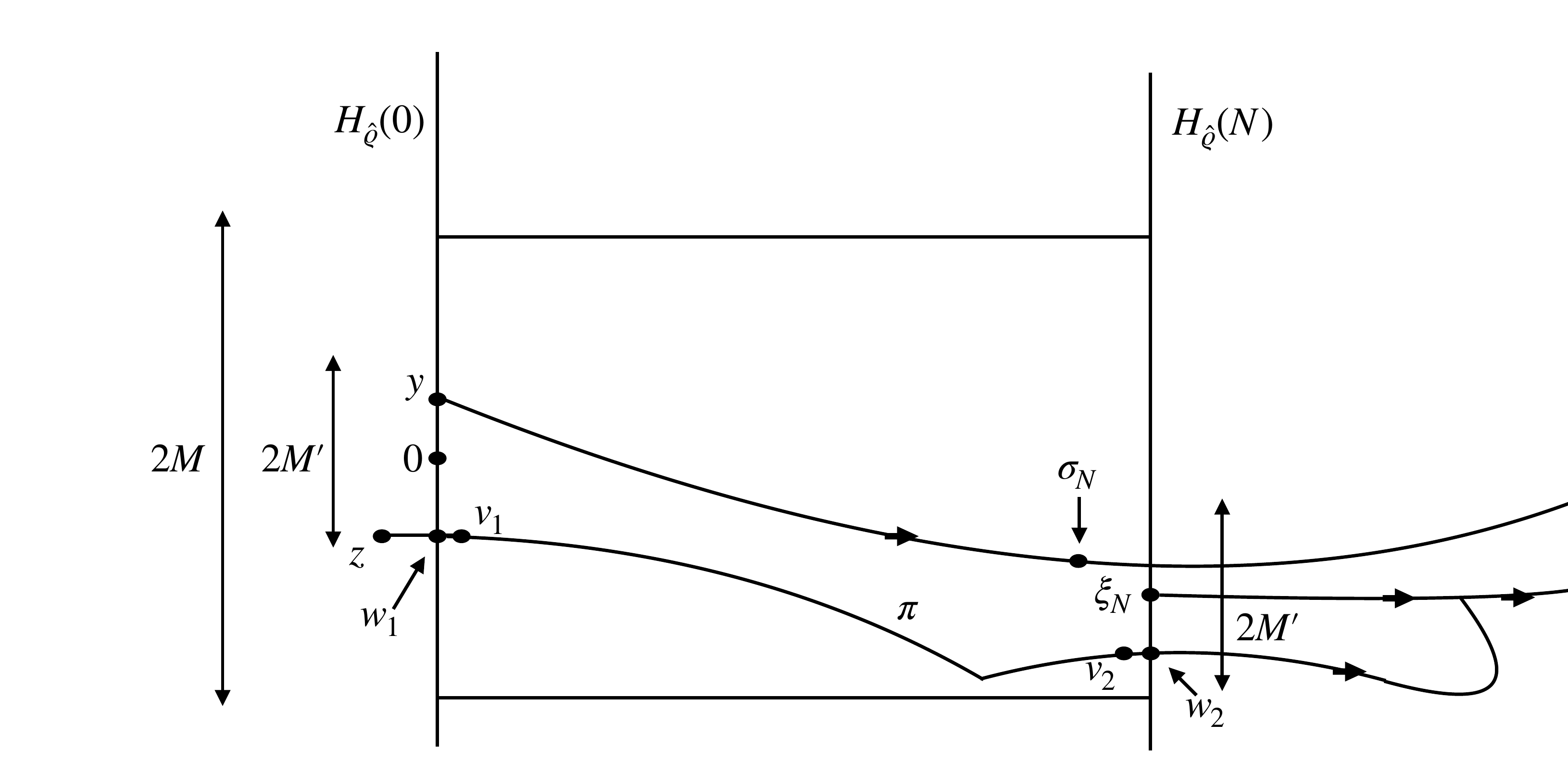}}
  \caption{Depiction of the proof of \eqref{eq: last_claim_1}. The passage time $T(v_1,v_2)$ is bounded from above by moving from $v_1$ to $y$, then to $\sigma_N$, then $\xi_N$, then to $v_2$. If we assume that $\Gamma_z$ intersects $\Gamma_{\xi_N}$, then items A2'.4(a-c) imply that $v_1$ must be close to $y$, and $v_2$ must be close to $\xi_N$. Because $\sigma_N$ is chosen on $\Gamma_y$ to be within distance $\epsilon \|\xi_N\|_1$ of $\xi_N$, these considerations lead to the upper bound for $T(v_1,v_2)$ displayed in \eqref{eq: final_upper_bound}.}
  \label{fig: strip_argument}
\end{figure}

To prove \eqref{eq: last_claim_1}, we argue by contradiction. The reader is encouraged to consult Figure~\ref{fig: strip_argument} throughout the proof. Suppose that $z \in \mathbb{Z}^d$ with $z \cdot \vartheta \leq 0$ has $\Gamma_z \cap \Gamma_{\xi_N} \neq \emptyset$. Because of item A2'.1, $\Gamma_z$ must contain a point whose dot product with $\vartheta$ is $\geq N$. This means that $\Gamma_z$ must cross the strip between $H_{\vartheta}(0)$ and $H_{\vartheta}(N)$, and so we may define $w_1$ as the last intersection (in $\mathbb{R}^d$) with $H_{\vartheta}(0)$ and $w_2$ as the first intersection (in $\mathbb{R}^d$) with $H_{\vartheta}(N)$ after $w_1$. Let $v_1$ be the first vertex of $\Gamma_z$ after $w_1$ and let $v_2$ be the last vertex of $\Gamma_z$ before $w_2$. (Note that $v_1=w_1$ and $v_2=w_2$ may occur.) Note that if $\pi$ is the segment of $\Gamma_z$ from $v_1$ to $v_2$, then
\begin{equation}\label{eq: pi_contain}
\pi \text{ is contained in the strip between } H_{\vartheta}(0) \text{ and } H_{\vartheta}(N).
\end{equation}

First we bound the passage time from $v_1$ to $v_2$ from above. By the triangle inequality,
\[
T(v_1,v_2) \leq T(v_1,y) + T(y,\sigma_N) + T(\sigma_N,\xi_N) + T(\xi_N, v_2).
\]
Here, $\sigma_N$ is a vertex of $\Gamma_y$ that is within $\ell^1$-distance $\epsilon\|\xi_N\|_1$ of $\xi_N$ (as provided by item A2'.2). For large $N$, we use the definition of $S$ and item A2'.3 to obtain the bound
\begin{equation}\label{eq: fancy_1}
T(v_1,v_2) \leq S\left( \|v_1-y\|_1 + \|\sigma_N-\xi_N\|_1 + \|\xi_N-v_2\|_1\right) + (S-\delta)\|\sigma_N-y\|_1.
\end{equation}
We use item A2'.4(a) to estimate
\begin{equation}\label{eq: fancy_2}
\|v_1-y\|_1 \leq \|v_1-w_1\|_1 + \|w_1-y\|_1 \leq d + \|w_1\|_1 + \|y\|_1 \leq d+2M',
\end{equation}
Similarly, by item A2'.4(b), 
\begin{equation}\label{eq: fancy_3}
\|\xi_N-v_2\|_1 \leq \|\xi_N - w_2\|_1 + \|w_2-v_2\|_1 \leq M' + d.
\end{equation}
Last, by the definition of $\sigma_N$,
\begin{equation}\label{eq: fancy_4}
\|\sigma_N - y\|_1 \leq \|\sigma_N - \xi_N\|_1 + \|\xi_N\|_1 + \|y\|_1 \leq (1+\epsilon)\|\xi_N\|_1 + M'.
\end{equation}
Combining \eqref{eq: fancy_2}-\eqref{eq: fancy_4} with the bound $\|\sigma_N- \xi_N\|_1 \leq \epsilon \|\xi_N\|_1$, and placing these in \eqref{eq: fancy_1}, we obtain
\[
T(v_1,v_2) \leq S(d+2M' + \epsilon\|\xi_N\|_1 + M'+d) + (S-\delta)((1+\epsilon)\|\xi_N\|_1 + M')
\]
Choosing $\epsilon$ small enough (recalling that $\delta$ is fixed), we find
\begin{equation}\label{eq: final_upper_bound}
T(v_1,v_2) \leq \left( S - \frac{3\delta}{4}\right) \|\xi_N\|_1,
\end{equation}
as long as $N$ is sufficiently large.

The bound \eqref{eq: final_upper_bound} implies that the edges of $\pi$ cannot be completely contained in $E_N$. If that were the case, then we would have
\[
T(v_1,v_2) \geq \lambda \|v_1-v_2\|_1 = \left( S- \frac{\delta}{2}\right) \|v_1-v_2\|_1,
\]
and by items A2'.4(a,b), we would have
\[
\|v_1-v_2\|_1 \geq \|\xi_N\|_1 - \|\xi_N - w_2\|_1 - \|w_2-v_2\|_1 - \|w_1\|_1 - \|w_1-v_1\|_1 \geq \|\xi_N\|_1 - 2(d+M'),
\]
so
\[
T(v_1,v_2) \geq \left( S - \frac{\delta}{2}\right) \left( \|\xi_N\|_1 - 2(d+M')\right).
\]
This would contradict \eqref{eq: final_upper_bound} for large $N$. We are forced to conclude that $\pi$ (which is constrained by \eqref{eq: pi_contain}) must therefore contain an edge that is not in $E_N$. By the definition of $E_N$, this edge is either (i) in $\Gamma_y$, or (ii) in $\Gamma_z$ for some $z$ satisfying at least one of the conditions A2'.4(a)-(c), or (iii) has an endpoint outside of $S(M,N)$. In case (i), since $\Gamma_y$ does not intersect $\Gamma_{\xi_N}$ by item A2'.2, we obtain a contradiction. In case (ii), we also obtain a contradiction since these $\Gamma_z$'s do not touch $\Gamma_{\xi_N}$. Finally, in case (iii), a contradiction comes from item A2'.4(c). This proves \eqref{eq: last_claim_1}.

Given \eqref{eq: last_claim_1}, we can use the mass transport principle to quickly move to a contradiction. Define a relation $\prec$ on $\mathbb{Z}^d$ as follows: we say $x \prec y$ (``$x$ precedes $y$'') if either (a) $x\cdot \vartheta < y \cdot \vartheta$ or (b) both $x \cdot \vartheta = y \cdot \vartheta$ and $x$ is less than or equal to $y$ in the lexicographic ordering (``dictionary ordering'') of $\mathbb{Z}^d \cap H_{\vartheta}(x \cdot \vartheta)$. Given a component $C$ of $\mathbb{G}$, we say that a vertex $x$ is the {\it progenitor} of $C$ if $x \in C$ and $x \prec y$ for all $y \in C$. Note that, due to \eqref{eq: last_claim_1}, for any outcome in $\mathsf{A}_2' \cap \{\Xi = E_N, t_e > \lambda \text{ for all } e \in E_N\}$, there is a unique progenitor of $C(\xi_N)$. Indeed, each vertex $w$ of $C(\xi_N)$ satisfies $w \cdot \vartheta > 0$, so there is a minimal number $r \in (0,N]$ such that $C(\xi_N)$ shares a vertex with $H_\vartheta(r)$. By item A2'.4(c), there are only finitely many such vertices, so one of them (the progenitor) is minimal in the lexicographic ordering. Motivated by this, we define the following mass transport function $m: \mathbb{Z}^d \times \mathbb{Z}^d \to \{0,1\}$:
\[
m(x,y) = \begin{cases}
1 & \quad \text{if } y \text{ is the unique progenitor of }C(x) \\
0 &\quad \text{otherwise}.
\end{cases}
\]
Similarly to \eqref{eq: mass_transport},
\begin{align}
%\begin{equation}\label{eq: final_quest}
1 \geq \sum_{x \in \mathbb{Z}^d} \mathbb{E}_\mu m(0,x) &= \sum_{x \in \mathbb{Z}^d} \mathbb{E}_\mu m(x,0) \nonumber \\
&= \mathbb{E}_\mu \left[ \sum_{x \in \mathbb{Z}^d} m(x,0) \mathbf{1}_{\{0 \text{ is the unique progenitor of }C(0)\}}\right]. \label{eq: final_quest}
%\end{equation}
\end{align}
By \eqref{eq: last_claim_1} and the above discussion,
\[
0 < \mu(\mathsf{A}_2' \cap \{\Xi = E_N, t_e > \lambda \text{ for all } e \in E_N\}) \leq \mu\left( \cup_x \{x \text{ is the unique progenitor of }C(x)\right),
\]
and so a union bound and translation invariance implies that $\mu(0 \text{ is the unique progenitor of }C(0)) > 0$. However, on this event, $\sum_{x \in \mathbb{Z}^d} m(x,0) = \sum_{x \in C(0)}m(x,0) = \infty$, and this contradicts \eqref{eq: final_quest} and shows that \eqref{eq: prefinal} cannot hold. However, \eqref{eq: prefinal} was a consequence of assumption \eqref{eq: main_assumption_for_contradiction}, so we finally conclude that \eqref{eq: main_assumption_for_contradiction} is false; that is, Theorem~\ref{thm: main_thm_restated} holds when $S<\infty$.

\subsection{Unbounded case}\label{sec: unbounded_case}

The unbounded case is actually easier than the bounded case so, in this section, we briefly indicate the (minimal) changes needed to cover it. We henceforth assume that $S$, the supremum of the support of the distribution of $t_e$, satisfies
\[
S=\infty.
\]
In this case, we proceed through the proof as before, but we change condition A2'.3 (and also A2.3) to:
\begin{enumerate}
\item[A2'.3'] for any vertices $v,w$ adjacent to both a vertex of $S(M,N)$ and a vertex of $S(M,N)^c$, there is a path connecting $v$ to $w$ which does not use edges that have both endpoints in $S(M,N)$, and which has passage time at most $\mathsf{C}$.
\end{enumerate}
By choosing the constant $\mathsf{C}$ large enough, depending on $M$ and $N$, it follows that $\mathsf{A}_2'$ satisfies the same statement as before: for $\epsilon>0$, there exist $y, M', \xi_N, \mathsf{C} = \mathsf{C}(M,N)$ with $\|y\|_1 \leq M'$ such that $\mu(\mathsf{A}_2') > 0$ for infinitely many $N$ and $M = M(N)$.

We continue through the proof until \eqref{eq: prefinal}, at which point we choose
\[
\lambda \text{ arbitrary instead of equal to } S - \frac{\delta}{2}.
\]
Now a version of Corollary~\ref{cor: our_events_modify} applies for all $\lambda$, since $\mathbb{P}(t_e> \lambda)>0$, and our new condition A2'.3' does not introduce complications (it involves passage times of paths in the complement of $S(M,N)$), so we can again conclude that
\[
\mu(\mathsf{A}_2', \Xi = E_N, t_e \geq \lambda \text{ for all } e \in E_N) > 0
\]
for infinitely many $N$ and $M = M(N)$, regardless of the value of $\lambda$. Thus we can choose $\lambda = \mathsf{C}+1$.

To obtain a contradiction, we follow Section~\ref{sec: completing}. After establishing claim \eqref{eq: last_claim_1}, the mass transport argument that completes the proof is the same as in the case $S<\infty$, so we focus on that claim. Defining $v_1$ and $v_2$ as before, note that both of these vertices are adjacent to a vertex of $S(M,N)$ and a vertex of $S(M,N)^c$, so item A2'.3' above implies
\[
T(v_1,v_2) \leq \mathsf{C}.
\]
This bound replaces \eqref{eq: final_upper_bound}. Next observe that, in this case, $\pi$ cannot contain {\it any} edges of $E_N$, since then we would have $T(v_1,v_2) =T(\pi) \geq \lambda = \mathsf{C}+1> \mathsf{C}$, a contradiction. Because $\pi$ must contain an edge that is not in $E_N$, we obtain a contradiction as before (following the last half of the paragraph below \eqref{eq: final_upper_bound}), and this shows \eqref{eq: last_claim_1}. This completes the sketch.

\appendix
\section{Modification lemma}\label{sec: appendix}

%{\color{green} make sure we refer to $H_\rho(\alpha)$ in the appendix (and elsewhere) instead of $H_{\hat{\varrho}}(\alpha)$.}

In this section, we give a general modification result which states that events $A \subset \Omega_1 \times \Omega_3$ which can be approximated in a certain sense by increasing cylinder events are stable under upward edge-modification. After, in Corollary~\ref{cor: our_events_modify}, we apply this to the events from Section~\ref{sec: modification_argument} to conclude inequality \eqref{eq: prefinal}.

We operate in the general setting of our geodesic measure $\mu$ on $\widetilde \Omega$; we write $(n_k)$ for the sequence such that $\mu_{n_k}^* \to \mu$, and recall that $S$ is the supremum of the support of $t_e$. (We allow for the possibility that $S=\infty$ in Theorem~\ref{thm: general_modification}.) We will think of events in the Borel sigma-algebra of $\Omega_1 \times \Omega_3$ as being on the full space $\widetilde \Omega$ by allowing the coordinate on $\Omega_2$ to be free. For the general result, we need some definitions. Recall that $\eta_\alpha$ refers to the third coordinate of the map $\Phi_\alpha$, from Section~\ref{sec: construction}.
\begin{definition}
\begin{enumerate}
\item We say that a Borel measurable $A \subset \Omega_1 \times \Omega_3$ is a graph-cylinder event if it is a finite union of events of the form $U \times V$, where $U \subset \Omega_1$ is Borel measurable and $V \subset \Omega_3$ is a cylinder event:
\[
V = \{\eta(e_1) = a_1, \dots, \eta(e_n) = a_n\}
\]
for fixed directed edges $e_i \in \vec{\mathcal{E}}^d$ and $a_i \in \{0,1\}$.
\item Given an edge $e_0 \in \mathcal{E}^d$, we say that a Borel measurable $A \subset \Omega_1 \times \Omega_3$ is $e_0$-approximable (for $\mu$) if there exists a sequence $(A_n)$ of graph-cylinder events such that
\[
\mu(A_n \Delta A) \to 0 \text{ as } n \to \infty
\]
and for each $n$, there exists $\alpha_0 = \alpha_0(n,e_0,A)$ such that if $\alpha \geq \alpha_0$ and $(t_e) \in \Omega_1$, then
\begin{equation}\label{eq: raise_it_up}
((t_e),\eta_\alpha((t_e))) \in A_n \text{ implies } ((t_e'),\eta_\alpha(t_e')) \in A_n
\end{equation}
whenever $(t_e') \in \Omega_1$ satisfies
\[
t_e' ~~~\begin{cases}
\geq t_{e_0} & \quad \text{if } e = e_0 \\
= t_e & \quad \text{if } e \neq e_0.
\end{cases}
\]
\end{enumerate}
\end{definition}

The main result states that $e_0$-approximable events have positive-probability upward modifications.
\begin{theorem}\label{thm: general_modification}
Let $e_0 \in \mathcal{E}^d$ and $\lambda \in [0,S)$. For any $r > 0$, there exists $s> 0$ such that if $A$ is an $e_0$-approximable event with $\mu(A) \geq r$, then
\[
\mu(A, t_{e_0} \geq \lambda) \geq s.
\]
\end{theorem}

\begin{proof}
We will need to approximate $A$ by graph-cylinder events, pull them back to $\Omega_1$ through the map $\Phi_\alpha$, and then perform the modification to the pull-back. Then we will push the modification back to $\Omega_1 \times \Omega_3$. To do this, we first must check: 
\begin{equation}\label{eq: cylinder_convergence}
\text{if }B \text{ is a graph-cylinder event then } \mu_{n_k}^\ast(B) \to \mu(B) \text{ as } k \to \infty.
\end{equation}
Note that there is no requirement for the first coordinate of the graph-cylinder event to lie in a cylinder set. We will be able to prove \eqref{eq: cylinder_convergence} at that level of generality because the marginal of $\mu_n^*$ on $\Omega_1$ is constant in $n$. 

Because graph-cylinder events can be written as a finite disjoint union of events of the form $U \times V$, where $U \subset \Omega_1$ is Borel measurable and $V \subset \Omega_3$ is a cylinder event, it suffices to check \eqref{eq: cylinder_convergence} for one such event. We will use the $\pi$-$\lambda$ theorem (see \cite[Sec.~2.1.1]{durrett}) and accordingly, we define, for our fixed cylinder $V$, the two collections
\begin{align*}
\Pi =~& \left\{ U \subset \Omega_1: U \text{ of the form } \{t_{e_i} \in [a_i,b_i] \text{ for } i = 1, \dots, n\}\right\}  \text{ and}\\
\Lambda =~& \left\{ U \subset \Omega_1 : \mu_{n_k}^*(U\times V) \to \mu(U\times V) \text{ as } k \to \infty\right\}.
\end{align*}
The sets in $\Pi$ are cylinder events that are induced by finite-dimensional closed rectangles; note that $\Pi$ generates the Borel sigma-algebra on $\Omega_1$. Furthermore, it is plain that $\Pi$ contains the empty set and is closed under finite intersections, so it is a $\pi$-system. Next, we check that $\Pi \subset \Lambda$, so let $U \in \Pi$. Since $\mu_{n_k}^* \to \mu$ weakly, we can invoke the Portmanteau theorem: we need only show that $\mu(\partial (U \times V)) = 0$, where $\partial (U\times V)$ is the metric boundary of $U\times V$ when viewing $\Omega_1 \times \Omega_3$ as a metric space. Because $\partial(U \times V) \subset \left( \partial U \times \Omega_3\right) \cup \left( \Omega_1 \times \partial V\right)$ and $V$ has empty boundary,
\[
\mu(\partial (U\times V)) \leq \mu(\partial U \times \Omega_3) = \mathbb{P}(\partial U) = \mathbb{P}(t_e = a_i \text{ or } b_i \text{ for some }i).
\]
Under {\bf A}, the weights are continuously distributed, so this probability is zero. Under {\bf B}, if there is $a \in \mathbb{R}$ such that $\mathbb{P}(t_e=a)>0$, then by translation invariance, there are two distinct edges $e$ and $f$ such that $\mathbb{P}(t_e=a \text{ and } t_f=a)>0$, and this contradicts uniqueness of passage times. Therefore $\mu(\partial (U\times V)) = 0$ and by the Portmanteau theorem, $U \in \Lambda$; that is, $\Pi \subset \Lambda$.

Next we prove that $\Lambda$ is a $\lambda$-system. It is plain that it contains the empty set and is closed under complements. Suppose that $(U_j)$ is a sequence of disjoint elements of $\Lambda$; we will prove that $U:=\cup_j U_j \in \Lambda$. For $\epsilon>0$, choose $J$ such that $\sum_{j =J+1}^\infty \mathbb{P}(U_j) < \epsilon/3$. Then letting $U^{(J)} = \cup_{j=1}^J U_j$ and $\bar{S} = S \times V$ for $S \subset \Omega_1$, we have
\begin{align*}
|\mu(\bar U)-\mu_{n_k}^*(\bar U)| &\leq |\mu(\bar U)-\mu(\bar{U}^{(J)})| + |\mu(\bar{U}^{(J)})-\mu_{n_k}^\ast(\bar{U}^{(J)})| + |\mu_{n_k}^\ast(\bar{U}^{(J)}) - \mu_{n_k}^\ast(\bar{U})| \\
&= \mu(\bar{U} \setminus \bar{U}^{(J)})+ |\mu(\bar{U}^{(J)}) - \mu_{n_k}^\ast(\bar{U}^{(J)})| + \mu_{n_k}^\ast(\bar{U} \setminus \bar{U}^{(J)} ) \\
&\leq 2\sum_{j=J+1}^\infty \mathbb{P}(U_j) + \sum_{j=1}^J |\mu(\bar{U_j}) - \mu_{n_k}^\ast(\bar{U_j})| \\
&\leq \frac{2}{3} \epsilon + \sum_{j=1}^J |\mu(\bar{U_j}) - \mu_{n_k}^\ast(\bar{U_j})|.
\end{align*}
Since $U_j \in \Lambda$ for each $j$, this is less than $\epsilon$ if $k$ is sufficiently large. We conclude that
\[
|\mu(U\times V) - \mu_{n_k}^\ast(U\times V)| = |\mu(\bar{U}) - \mu_{n_k}^\ast(\bar{U})| \to 0 \text{ as } k \to \infty;
\]
that is, $U \in \Lambda$. This implies that $\Lambda$ is a $\lambda$-system. By the $\pi$-$\lambda$ theorem, $\Lambda$ contains the sigma-algebra generated by $\Pi$, which means that $\mu_{n_k}^\ast(U\times V) \to \mu(U\times V)$ for all Borel $U\subset \Omega_1$ and cylinder $V \subset \Omega_3$. As we saw above, this implies \eqref{eq: cylinder_convergence}.

Now that we established \eqref{eq: cylinder_convergence}, we can apply results of \cite{DH14} to complete the proof of Theorem~\ref{thm: general_modification}. Let $A$ be $e_0$-approximable with $\mu(A) \geq r$ and let $(A_n)$ be the corresponding sequence of graph-cylinder events. Then write for any $n$ and $k$
\begin{align}
\mu(A,t_{e_0} \geq \lambda) &\geq \mu(A_n,t_{e_0} \geq \lambda) - \mu(A\Delta A_n) \nonumber \\
&\geq \mu_{n_k}^\ast(A_n,t_{e_0} \geq \lambda) - |\mu(A_n,t_{e_0}\geq \lambda) - \mu_{n_k}^\ast(A_n,t_{e_0} \geq \lambda)| - \mu(A\Delta A_n). \label{eq: step_1_approximate}
\end{align}
The definition of $\mu_{n_k}^\ast$ gives
\begin{equation}
\mu_{n_k}^\ast(A_n,t_{e_0} \geq \lambda) = \frac{1}{n_k} \int_0^{n_k} \mu_\alpha(A_n,t_{e_0} \geq \lambda)~\text{d}\alpha \label{eq: integral_end} \\
%&= \frac{1}{n_k} \int_0^{n_k} \mathbb{P}(((t_e),\eta_\alpha((t_e))) \in A_n, t_{e_0} \geq \lambda). 
\end{equation}
To estimate the integrand, which equals $\mathbb{P}(((t_e),\eta_\alpha((t_e))) \in A_n, t_{e_0} \geq \lambda)$, we use the modification result of \cite[Lem.~6.6]{DH14}. It states that under assumption {\bf A} or {\bf B}, for each $r> 0$, there exists $c = c(r)$ such that for all $e_0$-increasing events $A_0 \subset \Omega_1$ with $\mathbb{P}(A_0) \geq r/2$, one has $\mathbb{P}(A_0, t_{e_0} \geq \lambda) \geq c\mathbb{P}(A_0)$. (Here, $e_0$-increasing means that if $(t_e) \in A_0$ and $(t_e')$ is a configuration agreeing with $(t_e)$ off of $e_0$ and satisfying $t_{e_0}' \geq t_{e_0}$, then $(t_e') \in A_0$.) By the definition of $e_0$-approximable, the event $\{((t_e),\eta_\alpha((t_e))) \in A_n\}$ is $e_0$-increasing for $\alpha \geq \alpha_0$. Therefore
\begin{equation}\label{eq: cheesy_gordita}
\mu_\alpha(A_n) \geq \frac{r}{2} \text{ implies } \mu_\alpha(A_n,t_{e_0} \geq \lambda) \geq c\mu_\alpha(A_n) \text{ if } \alpha \geq \alpha_0.
\end{equation}

We last need to estimate the set of $\alpha$ such that the assumption of \eqref{eq: cheesy_gordita} holds. Because $\mu(A) \geq r$ and $\mu(A \Delta A_n) \to 0$ as $n \to \infty$, one has $\mu(A_n) \geq 7r/8$ for all $n$ greater than some $n_0$. Because $A_n$ is a graph-cylinder event, \eqref{eq: cylinder_convergence} gives $\mu_{n_k}^\ast(A_n) \to \mu(A_n)$ as $k \to \infty$, so $\mu_{n_k}^\ast(A_n) \geq 3r/4$ for all $n \geq n_0$ and for all $k$ greater than some $k_0 = k_0(n)$. The definition of $\mu_{n_k}^\ast$ then entails that if $n \geq n_0$ and $k \geq k_0$, then there is a set $S_{k,n}$ of Lebesgue measure at least $r n_k/4$ such that
\[
\alpha \in S_{k,n} \text{ implies } \mu_\alpha(A_n) \geq \frac{r}{2}.
\]
Combining this with \eqref{eq: cheesy_gordita}, we obtain for $n \geq n_0$ and $k \geq k_0$,
\[
\alpha \in S_{k,n} \text{ and } \alpha \geq \alpha_0 \text{ implies } \mu_\alpha(A_n, t_{e_0} \geq \lambda) \geq \frac{cr}{2}.
\]
Putting this back in \eqref{eq: integral_end}, for $n \geq n_0$ and $k \geq k_0$,
\[
\mu_{n_k}^\ast(A_n, t_{e_0} \geq \lambda) \geq \frac{cr}{2} \cdot \frac{1}{n_k} Leb(S_{k,n} \cap [\alpha_0,n_k]) \geq \frac{cr}{2} \left(\frac{r}{4} - \frac{\alpha_0}{n_k}\right).
\]
Last, we use this with \eqref{eq: step_1_approximate} to obtain for $n \geq n_0$ and $k \geq k_0$
\[
\mu(A,t_{e_0} \geq \lambda) \geq \frac{cr}{2}\left( \frac{r}{4} - \frac{\alpha}{n_k}\right)  - |\mu(A_n,t_{e_0}\geq \lambda) - \mu_{n_k}^\ast(A_n,t_{e_0} \geq \lambda)| - \mu(A\Delta A_n).
\]
Let $k \to \infty$ for fixed $n \geq n_0$ and use the fact that $A_n \cap \{t_{e_0} \geq \lambda\}$ is a graph-cylinder event (along with \eqref{eq: cylinder_convergence}) to find
\[
\mu(A,t_{e_0} \geq \lambda) \geq \frac{cr^2}{8} - \mu(A\Delta A_n).
\]
Last, let $n \to \infty$ to get $\mu(A,t_{e_0} \geq \lambda) \geq cr^2/8$. Putting $s = cr^2/8$ completes the proof of Theorem~\ref{thm: general_modification}.
\end{proof}

Having proved the general modification result, we move to proving inequality \eqref{eq: prefinal} from Section~\ref{sec: modification_argument}. We will assume throughout that $S<\infty$ for simplicity; similar arguments cover the case when $S=\infty$. 
%{\color{green} make sure that the $z$'s from definition of $\mathsf{A}_2'$ actually have $\Gamma_z$ with the correct properties in various places.}
\begin{corollary}\label{cor: our_events_modify}
With the notations of Section~\ref{sec: modification_argument}, for infinitely many $N$ and $M = M(N)$,
\[
\mu(\mathsf{A}_2', \Xi = E_N, t_e \geq \lambda \text{ for all } e \in E_N) > 0.
\]
\end{corollary}
\begin{proof}
Let $e_0, e_1, \dots, e_r$ be an enumeration of the edges in $E_N$. We modify the edges one by one, so we start by showing
\begin{equation}\label{eq: pasta_suprema}
\mu(\mathsf{A}_2', \Xi = E_N, t_{e_0} \geq \lambda) > 0.
\end{equation}
To use Theorem~\ref{thm: general_modification}, we must show that $\mathsf{A}_2' \cap \{\Xi = E_N\}$ is $e_0$-approximable. It will be easier to work with an approximation of $\mathsf{A}_2'$ which does not reference infinitely many $\Gamma_z$'s, so we define, for $k>0$, the event $\mathsf{A}_{2,k}'$ as follows. The conditions of A2'.1-A2'.3 remain unchanged, but A2'.4 is replaced by
\begin{enumerate}
\item[A2'.4$k$] $\Gamma_z \cap \Gamma_{\xi_N} = \emptyset$ for any $z \in \mathbb{Z}^d \cap [-k,k]^d$ which is an endpoint of an edge that contains a point $w \in \mathbb{R}^d$ with
\begin{enumerate}
\item[(a)] $w \in H_{\vartheta}(0)$ and $\|w\|_1 \geq M'$, or
\item[(b)] $w \in H_{\vartheta}(N)$ and $\|w-\xi_N\|_1 \geq M$, or
\item[(c)] $0 \leq w\cdot \vartheta \leq N$ and $w$ has Euclidean distance $\geq M$ from the line through $0$ and $\vartheta$.
\end{enumerate}
\end{enumerate}
We must also replace $\Xi$ by $\Xi_k = \Xi_k(M,N)$, the set of edges $e$ with both endpoints in $S(M,N)$ such that both
\begin{enumerate}
\item $e$ is not in $\Gamma_y$ and
\item $e$ is not in $\Gamma_z$ for any $z$ satisfying any of the conditions A2'.4$k$.
\end{enumerate}
Last, we generalize the definition of $\Gamma_x$ for $x \in \mathbb{Z}^d$ so that it is applicable to arbitrary elements $\eta \in \Omega_3$. To any such $\eta$, we can assign a directed graph $\mathbb{G}$ as before (induced by the directed edges $\langle u,v\rangle$ with $\eta(\langle u,v\rangle)=1$). We define $\Gamma_x$ to be the subgraph of $\mathbb{G}$ induced by the edges $\langle u,v\rangle$ such that $x \to u$ in $\mathbb{G}$. Note that $\mu$-a.s., this agrees with the old definition of $\Gamma_x$ because each $x$ has out-degree one. For this reason and the fact that $\mu$-a.s., directed paths in $\mathbb{G}$ are geodesics, we can replace the third item in the definitions of $\mathsf{A}_2$, $\mathsf{A}_2'$, and $\mathsf{A}_{2,k}'$ with the following: for all vertices $v \in \Gamma_y$ within $\ell^1$-distance $\epsilon\|\xi_N\|_1$ of $\xi_N$ and with $\|v-y\|_1 \geq M'$, there is a directed path $\pi_{y,v}$ in $\Gamma_y$ from $y$ to $v$ such that 
\[
T(\pi_{y,v}) \leq \|v-y\|_1 (S-\delta).
\]

Given these new definitions, we first show that for fixed $N$,
\begin{equation}\label{eq: k_approximation}
\mu\left(\left( \mathsf{A}_2'\cap \{\Xi = E_N\}\right) \Delta \left( \mathsf{A}_{2,k}' \cap \{ \Xi_k = E_N\}\right) \right) \to 0 \text{ as } k \to \infty.
\end{equation}
To do this, we use $\mathsf{A}_{2,k}' \supset \mathsf{A}_2'$ and $\Xi_k \supset \Xi$, and bound the left side by the sum
\begin{align}
&\mu\left( \mathsf{A}_2' \cap \{\Xi = E_N\} \cap \left( \mathsf{A}_{2,k}'\right)^c\right) + \mu\left( \mathsf{A}_2' \cap \{\Xi = E_N\} \cap \mathsf{A}_{2,k}' \cap \{\Xi_k = E_N\}^c\right) \nonumber \\
+~& \mu\left( \mathsf{A}_{2,k}' \cap \{\Xi_k = E_N\} \cap \left( \mathsf{A}_2'\right)^c\right) + \mu\left( \mathsf{A}_{2,k}' \cap \{\Xi_k = E_N\} \cap \mathsf{A}_2' \cap \{\Xi= E_N\}^c\right) \nonumber \\
\leq~& \mu\left( \mathsf{A}_2' \cap \{\Xi = E_N\} \cap \{\Xi_k = E_N\}^c\right) + \mu\left( \mathsf{A}_{2,k}' \setminus \mathsf{A}_2'\right) + \mu\left( \{\Xi_k = E_N\} \cap \mathsf{A}_2' \cap \{\Xi= E_N\}^c\right) \nonumber \\
\leq~& \mu\left( \mathsf{A}_{2,k}' \setminus \mathsf{A}_2'\right) + 2 \mu\left( \mathsf{A}_2' \cap \{\Xi_k \setminus \Xi \neq \emptyset\}\right). \label{eq: k_approximation_1}
\end{align}
The first term of \eqref{eq: k_approximation_1} converges to zero since $\mathsf{A}_2' = \cap_k \mathsf{A}_{2,k}'$. For the second term, note that $\Xi_1 \supset \Xi_2 \supset \dots \supset \Xi$ and $\cap_k \Xi_k = \Xi$, but since they are all subsets of a finite set, there must be a (random) $k_0$ such that $\Xi_k = \Xi$ for all $k \geq k_0$. This implies the second term converges to zero, and establishes \eqref{eq: k_approximation}.

Because of \eqref{eq: k_approximation}, it suffices to show that for any fixed $k$, the event $\mathsf{A}_{2,k}' \cap \{\Xi_k = E_N\}$ is $e_0$-approximable. For this we define the operator $\pi_n$ on $\Omega_1 \times \Omega_3$ by
\[
\pi_n((t_e),\eta) = ((t_e),\eta_{n,0}),
\]
where $\eta_{n,0}$ refers to the configuration which agrees with $\eta$ on all directed edges with both endpoints in the box $\Lambda_n = [-n,n]^d$ and which is equal to 0 for all other directed edges. Note that the graph $\mathbb{G}$ in $\pi_n((t_e),\eta)$ is a subgraph of the graph for $((t_e),\eta)$. (It is obtained by removing all directed edges outside $\Lambda_n$.) Then, as usual, we set $\pi_n^{-1}(A) = \{((t_e),\eta) : \pi_n((t_e),\eta) \in A\}$ for any measurable $A \subset \Omega_1 \times \Omega_3$. We will use these operators to construct graph-cylinder approximations. To this end, we prove that
\begin{equation}\label{eq: liminf}
\mathsf{A}_{2,k}' \cap \{\Xi_k = E_N\} = \bigcup_{m = 1}^\infty \bigcap_{n \geq m} \pi_n^{-1}\left( \mathsf{A}_{2,k}' \cap \{\Xi_k = E_N\}\right).
\end{equation}
We begin with the inclusion $\subset$. Suppose that $\mathsf{A}_{2,k}' \cap \{\Xi_k = E_N\}$ contains some $((t_e),\eta)$; we must show that for all large $n$, $\pi_n^{-1}\left( \mathsf{A}_{2,k}' \cap \{\Xi_k = E_N\}\right)$ contains $((t_e),\eta)$. Condition A2'.1 holds for $((t_e),\eta_{n,0})$ by the subgraph property. Similarly, in condition A2'.2, $\Gamma_y \cap \Gamma_{\xi_N} = \emptyset$ in $((t_e),\eta_{n,0})$ for each $n$. Since $\Gamma_y$ contains a vertex within $\ell^1$-distance $\epsilon \|\xi_N\|_1$ of $\xi_N$ in $((t_e),\eta)$, this is also true for $((t_e),\eta_{n,0})$ if $n$ is large enough. For A2'.3, since there exist paths $\pi_{v,y}$ in $\mathbb{G}$ for the configuration $((t_e),\eta)$ satisfying the bound $T(\pi_{v,y}) \leq \|v-y\|_1(S-\delta)$, and the collection of such $v$ is finite, this will also be true in $((t_e),\eta_{n,0})$ for large $n$. Next, A2'.4$k$ holds in $((t_e),\eta)$, it also holds in $((t_e),\eta_{n,0})$ by the subgraph property. 

We are left to show that $\Xi_k = E_N$ in $((t_e),\eta_{n,0})$, so long as $n$ is large. We will do this by showing that for all large $n$, $\Xi_k$ in $((t_e),\eta)$ equals $\Xi_k$ in $((t_e),\eta_{n,0})$. If a directed edge $\langle u,v\rangle$ with $u,v \in S(M,N)$ is not in $\Xi_k$ in $((t_e),\eta)$, then there exists a directed path in some $\Gamma_z$ from items A2'.4$k$(a-c) from $z$ to $u$. This path will also be in $\Gamma_z$ in $((t_e),\eta_{n,0})$ so long as $n$ is sufficiently large, so $\langle u,v \rangle \notin \Xi_k$ in $((t_e),\eta_{n,0})$ for $n$ large enough. Conversely, if $\langle u,v \rangle \notin \Xi_k$ in $((t_e),\eta_{n,0})$ for any given $n$, it is not in $\Xi_k$ in $((t_e),\eta)$ by the subgraph property.

For the inclusion $\supset$, the argument is similar. If A2'.1-4 hold in $((t_e),\eta_{n,0})$ for all large $n$, it is straightforward to check that they hold in $((t_e),\eta)$. Furthermore, the arguments of the last paragraph show that if $\Xi_k = E_N$ in $((t_e),\eta_{n,0})$ for all large $n$, then $\Xi_k = E_N$ in $((t_e),\eta)$. This completes the proof of \eqref{eq: liminf}.

%{\color{green} two problems: (a) replace $\mathsf{P}_n$ with: every path from any point in $[-k,k]^d \cup S(M,N)$ to $\Lambda_n^c$ has passage time larger than $S$ times the $\ell^1$-diameter of this set and (b) deal with the fact that the set of $\Gamma_z$ in the proof of increasing changes when we truncate, etc.}

Having established the equality in \eqref{eq: liminf}, we unfortunately need one last condition to construct graph-cylinder approximations. For any $n$, let $\mathsf{P}_n$ be the event
\[
\mathsf{P}_n = \left\{ \begin{array}{c}
\text{no geodesic from } z \text{ to } H_{\rho}(\alpha) \text{ touches }e_0 \text{ after exiting } \Lambda_n \\
 \text{ for all } \alpha \in \mathbb{R} \text{ and }z = \xi_N,y, \text{ or } z \text{ satisfying any of the conditions A2'.4}k \end{array} \right\} \cap \mathsf{Q},
\]
where $\mathsf{Q}$ is the event that $\lim_{m \to \infty} T(0,\mathbb{Z}^d \setminus \Lambda_m) = \infty$. 
%Here, the conditions A2'.4$k$ are understood to use the graphs $\Gamma_z(\alpha)$ in place of $\Gamma_z$. 
Note that $\mathsf{P}_n$ is measurable relative to the $\Omega_1$ coordinate, $(t_e)$, and on this event, there exists a geodesic between any two points (and between any point and any hyperplane). Furthermore, by the shape theorem \eqref{eq: shape_theorem}, one has $\mu\left( \cup_{m=1}^\infty \cap_{n \geq m} \mathsf{P}_n\right) = 1.$ Combining this with \eqref{eq: liminf}, we obtain
\begin{equation}\label{eq: liminf_2}
\mathsf{A}_{2,k}' \cap \{\Xi_k = E_N\} = \bigcup_{m = 1}^\infty \bigcap_{n \geq m}\left( \mathsf{P}_n \cap \pi_n^{-1}\left( \mathsf{A}_{2,k}' \cap \{\Xi_k = E_N\}\right)\right) ~\mu\text{-a.s.,}
\end{equation}
meaning that the symmetric difference of the left and right sides has $\mu$-measure zero. We use the right side to construct graph cylinder approximations. Write $B_n = \mathsf{P}_n ~\cap~ \pi_n^{-1}\left( \mathsf{A}_{2,k}' \cap \{\Xi_k = E_N\}\right)$, so that the right side of \eqref{eq: liminf_2} is $\cup_{m=1}^\infty \cap_{n \geq m} B_n$. From countable additivity, given $\varepsilon>0$, we can choose $m_0,m_1$ such that $m_0 \leq m_1$ and
\begin{equation}\label{eq: cylinder_construction}
\mu\left( \left( \mathsf{A}_{2,k}' \cap \{\Xi_k = E_N\}\right) \Delta \left( \bigcap_{n=m_0}^{m_1} B_n \right) \right) < \varepsilon.
\end{equation}
The finite intersection above is a graph-cylinder event, so we must show that it satisfies \eqref{eq: raise_it_up} in the definition of $e_0$-approximable. For this, it is enough to show that for a fixed $n$, there is $\alpha_0$ such that if $\alpha \geq \alpha_0$, then
\begin{equation}\label{eq: ramrod}
((t_e), \eta_\alpha((t_e))) \in B_n \text{ implies } ((t_e'), \eta_\alpha((t_e'))) \in B_n
\end{equation}
whenever $(t_e')$ is as in the definition of $e_0$-approximable.

We will show \eqref{eq: ramrod} for all $\alpha$, and to do this, we first define $\Gamma_x(\alpha) = \Gamma_x(\alpha)((t_e))$ to be the directed graph $\Gamma_x$ in the configuration $((t_e),\eta_\alpha((t_e)))$ and $\Gamma_{x,n}(\alpha)$ to be the directed graph $\Gamma_x$ in the configuration $\pi_n((t_e),\eta_\alpha((t_e)))$. The latter graph is the one obtained from stopping any directed path in $\Gamma_x(\alpha)$ once it first intersects $\Lambda_n^c$. We note here that 
\begin{equation}\label{eq: geo_truncation_1}
\begin{array}{c}
\text{if }((t_e),\eta_\alpha((t_e))) \in B_n, \text{ then } e_0 \notin \Gamma_z(\alpha) \\
\text{ for } z=\xi_N,y,\text{ or } z \text{ satisfying any of the conditions A2'.4}k.
\end{array}
\end{equation}
%Again, the conditions A2'.4$k$ use the graphs $\Gamma_z(\alpha)$. 
Indeed, supposing that $((t_e),\eta_\alpha((t_e))) \in B_n$, then by definition of $\Xi_k$, we have $e_0 \notin \Gamma_{z,n}(\alpha)$ for $z = \xi_N,y$ or $z$ satisfying any of the conditions A2'.4$k$. However if any such $z$ had $e_0 \in \Gamma_z(\alpha)$, this would mean that some geodesic from $z$ to $H_{\rho}(\alpha)$ exits $\Lambda_n$ and then returns to an endpoint of $e_0$. This would mean $(t_e) \notin \mathsf{P}_n$, a contradiction.

%To do this, take $\alpha_0$ so large that $[-k,k]^d$, $[-n,n]^d$, and $S(M,N)$ are on one side of $H_{\hat{\varrho}}(\alpha_0)$. 
Now, for \eqref{eq: ramrod}, suppose that $((t_e), \eta_\alpha((t_e))) \in B_n$. To show that $((t_e'), \eta_\alpha((t_e'))) \in B_n$, first observe that increasing $t_{e_0}$ to $t'_{e_0}$ does not change any $\Gamma_z(\alpha)$ for $z$ listed in \eqref{eq: geo_truncation_1}. Using this, we move to condition A2'.1$k$. Because $\Gamma_{\xi_N}(\alpha)$ is the same in $(t_e)$ and $(t_e')$, since $\Gamma_{\xi_N,n}(\alpha)((t_e))$ obeys A2'.1$k$, so does $\Gamma_{\xi_N,n}(\alpha)((t_e'))$. Similar reasoning applied to $y$ and for $z$ listed in A2'.4$k$ shows that this and A2'.2$k$ hold for $((t_e'),\eta_\alpha((t_e')))$. For A2'.3$k$, we also use that $\Gamma_y(\alpha)$ does not change from $(t_e)$ to $(t_e')$, along with the fact that no edges in any $\pi_{y,v}$ can contain $e_0$ (since they are contained in $\Gamma_{y,n}(\alpha)$, which is a subgraph of $\Gamma_y(\alpha)$, and this latter graph does not contain $e_0$ by \eqref{eq: geo_truncation_1}). We conclude that A2'.3$k$ holds for $((t_e'),\eta_\alpha((t_e')))$.

Last we must prove that $\mathsf{P}_n \cap \{\Xi_k = E_N\}$ holds for $((t_e'),\eta_\alpha((t_e')))$. As noted before, because of \eqref{eq: geo_truncation_1}, all geodesics from any $z$ listed there to $H_{\rho}(\alpha)$ remain the same in the configuration $(t_e')$. Because none of them contained $e_0$ in $(t_e)$, they still do not in $(t_e')$ and so $\mathsf{P}_n$ holds trivially in $((t_e'), \eta_\alpha((t_e')))$. The second, the event $\mathsf{Q}$, is increasing in the edge-weights, so since it holds for $(t_e)$, it also holds for $(t_e')$. We are left to show that $\Xi_k = E_N$ in $((t_e'), \eta_\alpha((t_e')))$. However, $\Xi_k$ in $((t_e),\eta_\alpha((t_e)))$ is the set of edges with both endpoints in $S(M,N)$ that are not in $\Gamma_{y,n}(\alpha)$ or $\Gamma_{z,n}(\alpha)$ for any $z$ satisfying any of the conditions A2'.4$k$ in the configuration $(t_e)$. As we saw, these graphs do not change when we increase $t_{e_0}$ to $t_{e_0}'$, so this set remains the same in $((t_e'),\eta_\alpha((t_e')))$. This completes the proof of \eqref{eq: ramrod}, and shows that \eqref{eq: pasta_suprema} holds.

To complete the proof of Corollary~\ref{cor: our_events_modify}, we extend \eqref{eq: pasta_suprema} to all edges in $E_N$ by induction. for $s = 0, \dots, r$, let S($s$) be the statement:
\begin{itemize}
\item[S($s$):]  $\mu(\mathsf{A}_2', \Xi = E_N, t_{e_p} \geq \lambda \text{ for } p = 0, \dots, s) > 0$ for infinitely many $N$ and $M=M(N)$.
\end{itemize}
We have already shown that S(0) is true. Assuming that S($s$) is true for some $s= 0, \dots, r-1$, we show that S($s+1$) is true. Running the argument leading to \eqref{eq: cylinder_construction} with $e_{s+1}$ in place of $e_0$, we obtain $(A_n)$, a sequence of graph-cylinder events such that
\[
\mu\left( \left( \mathsf{A}_2' \cap \{\Xi = E_N\}\right) \Delta A_n \right) \to 0,
\]
and
\[
((t_e),\eta_\alpha((t_e))) \in A_n \text{ implies } ((t_e'), \eta_\alpha((t_e'))) \in A_n
\]
whenever $(t_e')$ agrees with $(t_e)$ off $e_{s+1}$ and satisfies $t_{e_{s+1}}' \geq t_{e_{s+1}}$. Set
\[
A_n' = A_n \cap \{t_{e_p} \geq \lambda \text{ for } p = 0, \dots, s\}.
\]
Then
\[
\mu\left( \left( \mathsf{A}_2' \cap \{\Xi = E_N\} \cap \{t_{e_p} \geq \lambda \text{ for } p = 0, \dots, s\}\right) \Delta A_n' \right) \to 0.
\]
By Theorem~\ref{thm: general_modification}, we must show only that for $(t_e),(t_e')$ as above (they agree off $e_{s+1}$ but $t_{e_{s+1}}' \geq t_{e_{s+1}}$), one has
\[
((t_e),\eta_\alpha((t_e))) \in A_n' \text{ implies } ((t_e'), \eta_\alpha((t_e'))) \in A_n'.
\]
But this follows immediately: since $((t_e), \eta_\alpha((t_e))) \in A_n$, so is $((t_e'),\eta_\alpha((t_e')))$, and the weights $t_{e_p}$ for $p=0, \dots, s$ do not change from $(t_e)$ to $(t_e')$, so we obtain $((t_e'),\eta_\alpha((t_e'))) \in A_n'$. We find that the event $\mathsf{A}_2' \cap \{\Xi = E_N\} \cap \{t_{e_p} \geq \lambda \text{ for } p = 0, \dots, s\}$ is $e_{s+1}$-approximable, and since it has positive probability by S($s$), we can apply Theorem~\ref{thm: general_modification} to conclude that S($s+1$) is true. This completes the proof of Corollary~\ref{cor: our_events_modify}.

%Since S(r) is true, we can find a sequence of graph-cylinder events $(A_n)$ such that
%\[
%\mu\left( \left( \mathsf{A}_2', \Xi = E_N, t_{e_p} \geq \lambda \text{ for } p = 0, \dots, s-1\right) \Delta A_n\right) \to 0,
%\]

\end{proof}

\bigskip
\noindent
{\bf Acknowledgements.} The research of M.~D. is supported by an NSF CAREER grant. The research of J.~H. is supported by NSF grant DMS-161292, and a PSC-CUNY Award, jointly funded by The Professional Staff Congress and The City University of New York.

\end{document}